\documentclass{article}

\usepackage{lineno}
\modulolinenumbers[5]

\usepackage[margin=1in]{geometry}

\usepackage{amsmath}
\usepackage{amsfonts}
\usepackage{amssymb}
\usepackage{amsthm}
\usepackage{color}  
\usepackage{graphicx}
\usepackage{bm}
\usepackage{mathrsfs}
\RequirePackage[colorlinks,citecolor=blue,urlcolor=blue]{hyperref}

\usepackage{comment}

\usepackage{varwidth}
\usepackage{caption}
\usepackage{tabularx}
\usepackage[export]{adjustbox}
\usepackage{multirow}
\usepackage{booktabs}
\usepackage{appendix} 

\theoremstyle{definition}
\newtheorem{assumptionalt}{Assumption}

\numberwithin{equation}{section}

\theoremstyle{plain}
\newtheorem{lemma}{Lemma}[section]
\newtheorem{proposition}[lemma]{Proposition}
\newtheorem{theorem}[lemma]{Theorem}
\newtheorem{corollary}[lemma]{Corollary}

\theoremstyle{definition}
\newtheorem{definition}[lemma]{Definition}

\newenvironment{assumption}[1]{\renewcommand\theassumptionalt{#1}\assumptionalt}{\endassumptionalt}

\newtheorem{remark}[lemma]{Remark}
\newtheorem{example}[lemma]{Example}

\def\P{{\mathbb P}} 
\def\R{{\mathbb R}}
\def\E{{\mathbb E}}
\def\N{{\mathbb N}}

\DeclareMathOperator{\Dom}{Dom}
\DeclareMathOperator{\Cov}{Cov}

\providecommand{\eps}{\varepsilon}

\renewcommand{\theta}{\vartheta}

\begin{document}


\title{Parameter estimation from local measurements for a class of stochastic Burgers equations}

\author{Josef Jan\'ak\\ University of Pavia,  Italy\\ Email: josefjanak@seznam.cz \and Enrico Priola\\ University of Pavia, Italy\\Email: enrico.priola@unipv.it}

\maketitle

\begin{abstract} 
We deal with a class of semilinear SPDEs driven by space–time white noise that includes the one dimensional stochastic Burgers equation. Such equations can have nonlocal and quadratic nonlinearities. We consider the problem of estimation of the diffusivity parameter in front of the second-order spatial derivative. Based on local observations in space, we study the estimator derived in ({\it Ann. Appl. Probab.} {\bf 31} (2021) 1--38) for linear sto\-chas\-tic heat equation that has also been used in ({\it Bernoulli} {\bf 29} (2023) 2035--2061) to cover certain class of semilinear SPDEs including stochastic Burgers equations driven by trace class noise. The space-time white noise case we consider has also relevant physical motivations. After we establish new regularity results for the solution, we are able to show that our proposed estimator is strongly consistent and asymptotically normal.
\end{abstract}
 
\noindent {\small {\it Keywords:} 
Generalized stochastic Burgers equation; space-time white noise; local meas\-ure\-ments; diffusivity estimation; augmented MLE; central limit theorem. \\
2020 MSC: 60H15, 62F12, 35R60}
 


\section{Introduction}

We consider estimation of the diffusivity parameter $\theta>0$ in a class of generalized stochastic Burgers equation driven by space-time white noise:
\begin{equation} \label{eq:Burgers}
\begin{cases}
dX(t) &= \theta \partial_{xx}^2 X(t) \, dt + \frac{1}{2} \partial_x \left( X^2(t) \right) \, dt + F(X(t)) \, dt + dW(t), \\
X(0) &= X_0, \\
X(t)|_{\partial \Lambda} &= 0, \quad 0 < t \leq T.
\end{cases} 
\end{equation}
Here $W$ stands for a cylindrical Wiener process on $L^2(\Lambda)$, where $\Lambda = (0,1) \subset \R$, $dW(t)/dt$ is also referred to as space-time white noise. The nonlinearity $F$ is locally Lipschitz and can have a quadratic growth. When $F=0$ we obtain the (classical) stochastic Burgers equation. We consider Dirichlet boundary conditions and a continuous and  deterministic initial condition $X_0$ defined on $\bar \Lambda$. Note that \eqref{eq:Burgers} is often written as
\begin{equation} \label{bur0}
\begin{cases} 
dX(t,x) &= \theta \frac{\partial^2}{\partial x^2} X(t, x) \, dt + \frac{1}{2} \frac{\partial}{\partial x} X^2(t, x) \, dt + F(X(t, \cdot))(x) \, dt + dW_t(x), \\
X(0,x) &= X_0(x), \quad x \in (0, 1) = \Lambda, \\
X(t,0) &= X(t,1) = 0, \quad 0 < t \leq T,
\end{cases}
\end{equation}
and the same equation with the nonlinear term $- \frac{1}{2} \frac{\partial}{\partial x} X^2(t, x) \, dt$ can be treated similarly. The stochastic Burgers equation has an important role in fluid dynamics and has been studied by several authors. We mainly refer to the works \cite{DPDT} and \cite{gatarekDap}, where the existence and uniqueness of the global solution as well as the existence of an invariant measure has been established; see also Chapter 14 from \cite{DPZ96}. We also  recall \cite{DPD}, where moment estimates for the solution were established. Generalized stochastic Burgers equations like \eqref{eq:Burgers} are studied in the literature under different assumptions (see, for instance, \cite{gyongy}, \cite{gyongyrovira}, \cite{kim}). Interesting applications of generalized Burgers equations are mentioned in \cite{tersenov}. 

We first establish existence and uniqueness of regular solutions to \eqref{eq:Burgers} (see Proposition  \ref{prop:regulXtilde}) and prove moment estimates (or integrability properties) for the solution (see Proposition \ref{prop:L2 integrability}). The following one is  an example of $F$ we can treat:
$$ 
F(u)(x) = - c_0 \, u(x) |u(x)| \, + \, u(x) \, h \big( \| u\|_{L^2(\Lambda)} \big), \quad x \in \Lambda, \, a.e.,
$$
$u \in L^2(\Lambda)$, $h$ is any bounded $C^1$-function, $c_0 > 0$ (see also hypothesis \eqref{bb} and Examples \ref{11}, \ref{12} and \ref{c44}). It seems that even well-posedness of \eqref{eq:Burgers} under our assumptions is not covered by the results in the literature. In particular our generalization of the stochastic Burgers equation leads in a different way than \cite{gyongy}; see also Remark \ref{gy}. We mention other types  of generalized stochastic Burgers equations in Remark \ref{open}.   
  
The estimation of the diffusivity (or the drift part) has become a standard inference problem for stochastic partial differential equations (SPDEs) {in both parametric and nonparametric setting, but it is worth noting that parameter estimation for the stochastic Burgers equation has not been much investigated. In that regard, we may mention the work \cite{CGH}, where statistical inference for the very related Navier-Stokes equations is studied, and \cite{PS}, where stochastic Burgers equation driven by a~trace class noise fits into a class of semilinear stochastic evolution equations.

In general, the classical methods} mostly rely on the observations in the Fourier space over some time interval (with the number of modes $N \rightarrow \infty$ in the so-called spectral approach, see, e.g., \cite{HR}) or on the ergodic properties of the solution (with $T \rightarrow \infty$, {see, e.g., \cite{Ku}}) or on discrete sampling (see, e.g., \cite{BT19}, \cite{BT20} and \cite{HT}).  However in this paper we continue to expand the methodology that is based on local measurements. This novel approach was first introduced in \cite{AR} and relies on a local observation in space of the solution against some small ``observational window''. That is represented by a function called kernel $K$ that is scaled and localized around certain observation point $x_0 \in \Lambda$. The observational data then comes in the form of convolution $\left\langle X(t), K_{\delta, x_0} \right\rangle_{L^2(\Lambda)}$ and the asymptotics is studied for $\delta \rightarrow 0^+$.

The augmented maximum likelihood estimator (augmented MLE) of $\theta$ was introduced in \cite{AR} for linear SPDEs in nonparametric setting. Since then, it was used in the parametric setting to experimental data from cell biology (see \cite{ABJR} and note that such usage is the practical motivation of local measurements) and from the theoretical point of view, the estimator is remarkably flexible and robust to some misspecifications. For instance, \cite{JR} studied the augmented MLE and its variants in the case of stochastic heat equation driven by the multiplicative noise, while the authors in \cite{ACP} focused on a case of semilinear SPDEs. Their work covers a wide range of nonlinearities (such as stochastic reaction-diffusion equations) assuming that the SPDE is driven by additive noise $B \, dW(t)$, where $B^*$ scales as the fractional operator $(- \partial_{xx}^2)^{-\gamma}$. However, the classical stochastic Burgers equation with $F=0$ in \eqref{eq:Burgers} is still in some sense a~``borderline case'' of their presented theory. The augmented MLE is strongly consistent, but the asymptotic normality needed to be analyzed separately and only in the case of a trace class noise: $(- \partial_{xx}^2)^{- \gamma} \, dW(t)$ for $\gamma > 1/4$ (that is the case when $B$ is Hilbert-Schmidt). This was the starting point of our work.  

In particular we study the augmented MLE for classical stochastic Burgers equation  that is driven by space-time white noise (i.e., the case when $\gamma = 0$). We believe that this case is  challenging and from the theoretical point of view also very interesting (for physical motivations see \cite{bertini}, \cite{1} and \cite{2}; for numerical approximations see \cite{hairer}). When $F=0$ our results and arguments can be further extended to cover the case of spatially ``smoothed'' additive noise $(- \partial_{xx}^2)^{- \gamma} \, dW(t)$ with $\gamma \in [0, 1/4]$ but, in fact, to cover all $\gamma \geq 0$. However, this  would require extra effort. Among other things, one should revise all the proofs in  Appendix for such general $\gamma$, $d = 1$ and $B = (- \partial_{xx}^2)^{- \gamma}$. 

At first, we consider the notion of mild solution to our equation \eqref{eq:Burgers} using so-called ``splitting technique'', where we split the solution to the linear part $\bar{X}$ and the nonlinear part $\widetilde{X}$. For the analysis of a (possible) asymptotic bias, we need to assert the regularities of both parts. The results on regularity of the stochastic convolution $\bar{X}$ are basically known, but for the nonlinear part $\widetilde{X}$, the results seem to be new; see in particular  Propositions \ref{prop:regulXtilde} and \ref{prop:L2 integrability}.
  
For the analysis of the error $\delta^{-1} \left( \hat{\vartheta}_{\delta} - \vartheta \right)$, we follow general ideas from \cite{ACP} and its supplement, where many auxiliary lemmas (e.g., for scaling) and propositions do hold true even in our case $\gamma = 0$ (see Section \ref{subsection:resultsACP} in Appendix). It must be noted though, that we use different proof techniques to adjust to our space-time white noise case  and to clarify some points of the original paper; see also Remark \ref{altri}. We believe that Lemma \ref{lemma:U3} is pivotal to our parameter estimation. In its proof we show  the expansion of the non-Gaussian process
$\widetilde{X}(t, x_0)= \widetilde{X}(t) (x_0)$ in the form $\sum_{i=1}^{\infty} b_i(t) \nu_i$. That requires subtle care, especially in building a~separable Hilbert space $L^2(\Omega, \mathcal F', \P)$ with its orthonormal basis $(\nu_i)_{i=1}^{\infty}$ (we are considering  a suitable $\sigma$-field $\mathcal F'\subset \mathcal F$). Note that we are not able to apply the classical Karhunen-Lo\`eve expansion in our case due to the low regularity of $\widetilde{X}$ (see also  Remark \ref{Karhunen}).

In conclusion, we find that the augmented MLE is asymptotically normal, asymptotically unbiased and with the rate of convergence and the asymptotic variance that align to previous results from the theory of estimation using local measurements. The precise statement is formulated in Theorem \ref{thm:main}. Based on the asymptotic result, we can deduce (data-driven) confidence intervals for $\theta$. We mention that numerical simulations for stochastic Burgers equation \eqref{eq:Burgers} driven by space-time white noise are proposed in Section 4 of \cite{ACP} (see Figure 1 therein, where the white noise is multiplied by a factor $\sigma = 0.05$).

Our paper is organized as follows. The notation and the exact setting is introduced in Section \ref{sec:model}, where we also discuss the notion of  mild and weak solution. The results on the regularity of the nonlinear part $\widetilde{X}$ can be also found here. In Section \ref{sec:mainresults} we construct the estimator, revisit assumptions of the model and present the main asymptotic results. The proofs and auxiliary results are delegated to Section \ref{sec:proofs} and Appendix, where they are assorted by the topic. For the sake of completeness, we also provide the most relevant results from \cite{ACP} that are revisited and adjusted to our setting in Section \ref{subsection:resultsACP} in Appendix.

\section{The model}\label{sec:model}

\subsection{Notation}\label{subsection:notation}
Consider $\Lambda = (0,1)$ and the space $L^2(\Lambda)$ equipped with the usual $L^2$-norm $\| \cdot \| := \| \cdot \|_{L^2(\Lambda)}$ and the scalar product $\left\langle \cdot, \cdot \right\rangle := \left\langle \cdot, \cdot \right\rangle_{L^2(\Lambda)}$. Even though the set $\Lambda$ is only one dimensional, we use the standard Laplace operator notation 
$$
\Delta u = \partial_{xx}^2 u = u''
$$
for a function $u$ which is regular enough. We denote $(\lambda_n, e_n)_{n=1}^{\infty}$ the eigensystem of the positive, self-adjoint operator $- \Delta$ on $\Lambda$ with Dirichlet boundary conditions. That is
$$
\lambda_n = \pi^2 n^2, \quad e_n(x) = \sqrt{2} \sin(nx \pi), \quad x \in (0,1).
$$
To describe higher regularities, we consider for $s \geq 0$ the fractional Laplacians $(- \Delta)^{s/2}$ and denote their domains by $H^s = H^s(\Lambda)$ with the norm $\| \cdot \|_s := \| \cdot \|_{H^s} := \| (- \Delta)^{s/2} \cdot \|_{L^2(\Lambda)}$. The Hilbert spaces
$$ 
H^s = \Dom \left( (- \Delta)^{s/2}) \right) = \Big\{ u \in L^2(\Lambda); \, \sum_{n=1}^{\infty} \lambda_n^s u_n^2 = \sum_{n=1}^{\infty} \lambda_n^s \left\langle u, e_n \right\rangle^2 < \infty \Big\}
$$
are equipped with the scalar product $\left\langle u, v \right\rangle_{H^s} = \sum_{n=1}^{\infty} \lambda_n^s u_n v_n$ for $u, v \in H^s$. We also set
$$
H := L^2(\Lambda)
$$
and by $B(0,n)$ we denote the closed ball in the space $H$ with the radius $n$, that is $B(0,n) = \{ u \in H; \, \| u \| \leq n \}$.
Recall that $H^1 =H_0^1$ with equivalence of norms, where $H_0^1$ denotes  the space of all $u \in H$ with weak derivative $u' \in H$ and such that $u(x)=0$ for $x\in\partial\Lambda$ (cf. Section 1.3 in \cite{Henry}, or \cite{hairer}).

For $s \in \R$ and $p \ge 2$ we consider also $W^{s,p}(\Lambda) := \{u \in L^p(\Lambda); \, \| u \|_{s,p} < \infty \}$, where $\| \cdot \|_{s,p} := \| (-\Delta)^{s/2} \cdot \|_{L^p(\Lambda)}$. These are fractional Sobolev spaces associated to $\Delta$ defined as Bessel potentials spaces. The space $C(\bar{\Lambda})$ of all continuous functions on $\bar{\Lambda}$ is equipped with the supremum norm $\| \cdot \|_{\infty}$. The spaces $C^{k, \alpha}(\Lambda)$ for $k \geq 0$ integer and $0 < \alpha < 1$ are  H\"older spaces equipped with usual norms.

We fix a constant $\theta > 0$, that is the parameter of interest, and denote by $(S_{\theta}(t), t \geq 0)$ the strongly continuous semigroup on $L^2(\Lambda)$ generated by $\theta \Delta$.

Throughout this work we fix a finite time horizon $T > 0$ and let $(\Omega, \mathcal F, (\mathcal F_t)_{t \geq 0}, \mathbb P)$ be a stochastic basis (satisfying the usual hypotheses) with a~cylindrical Wiener process $W$ on $L^2(\Lambda)$ (cf. \cite{DPZ}). The process $W(t)$ is formally given by ``$W(t) = \sum_{n = 1}^{\infty} \beta_n(t) e_n$'' where $(\beta_n)$ are independent one di\-men\-sio\-nal Wiener processes adapted to the previous filtration; we are using the above orthonormal basis $(e_n)$ in $L^2(\Lambda)$.

Throughout, all equalities and inequalities, unless otherwise mentioned, will be understood in the $\P$-a.s. sense. By $A_\delta \lesssim B_\delta$ we mean that there exists some constant $C > 0$ such that $A_\delta \leq C B_\delta$ for all values $\delta$ under consideration. Here, we work with $\delta\in(0,1)$ or with the convergence $\delta\to 0$. Convergence in pro\-ba\-bi\-li\-ty and convergence in distribution are denoted by $\stackrel{\mathbb P}{\rightarrow}$ and $\stackrel{d}{\rightarrow}$, respectively. The symbol $A_\delta = O_{\mathbb P}(B_\delta)$ for random variables $A_{\delta}, B_{\delta}$ means that $A_\delta / B_\delta$ is tight, that is, $\sup_\delta \mathbb P (|A_\delta| > C |B_\delta|) \rightarrow 0$ as $C \rightarrow \infty$. The notation $A_\delta = o_{\mathbb P}(B_\delta)$ stands for $A_\delta / B_\delta \stackrel{\mathbb P}{\rightarrow} 0$ as $\delta \rightarrow 0$.

\subsection{The generalized stochastic Burgers equation}

We study the generalized stochastic Burgers equation \eqref{eq:Burgers} driven by a cylindrical 
Wiener proces $W$ on $H= L^2(\Lambda)$. Our nonlinearity is given by a function $F: L^2(\Lambda) \rightarrow L^1(\Lambda)$. We also assume the following conditions:
\begin{itemize}
\item[(a)] For any $ R > 0 $ there exists  $ C_R > 0$ such that for all $ u, v \in B(0,R)= \{ u \in H; \, \| u \| \leq R \}:$
$$
\| F(u) - F(v) \|_{L^1(\Lambda)} \leq C_R \| u - v \|.
$$
\item[(b)] There exists $q \geq 2, \, C_q > 0 $ such that for all $ u, v \in L^2(\Lambda)$ for any $ x \in \Lambda, \, a.e.:$  
\begin{equation*} 
F(u + v)(x) \cdot v(x) \leq C_q ( 1  + |u(x)|^q + |v(x)|^2 ).  \label{bb} \tag{F}
\end{equation*} 
\item[(c)] $F$ maps bounded sets of $L^{\infty}(\Lambda)$ into bounded sets of $L^p(\Lambda)$ for some $p > 2$. 
\end{itemize}
Note that hypotheses (a) and (b) are needed to prove well-posedness of our SPDEs as well as to establish the additional  moment estimates of the solution (see Proposition \ref{prop:L2 integrability}). On the other hand condition (c) is only needed to perform the parameter estimation (see in particular the proof of Lemma \ref{lemma:U4}).

\vskip 1mm The following stronger conditions imply the previous assumptions:
(i) $F: L^2(\Lambda) \to L^2(\Lambda)$ is locally Lipschitz, (ii) there exists $C > 0 $ such that for any $u \in L^2(\Lambda)$ and for all $x \in \Lambda, \, a.e.$, we have: $|F(u)(x)| \leq C \left(1 + |u(x)| \right)$. However these hypotheses prevent the study of nonlinearities which grow quadratically. 

\vskip 1mm
Now we provide examples of nonlinearity $F: L^2(\Lambda) \to L^1(\Lambda)$ that satisfy our conditions (a), (b) and (c). We give more details in Section \ref{subsection:examples} in Appendix. The first examples are nonlocal nonlinearities.
\begin{example} \label{11}
Let $f_1, f_2: \R \rightarrow \R$ be real functions such that $f_1$ is bounded and locally Lipschitz and $f_2$ is (globally) Lipschitz. For any $u \in L^{2}(\Lambda)$, we define
$$
(F(u))(x) = f_1 \left( \| u \| \right) \cdot f_2 \left( u(x) \right), \quad x \in \Lambda, \, a.e.. 
$$
\end{example}
\begin{example} \label{12}
Let $g: \R \rightarrow \R$ be a real function such that $g$ is bounded and locally Lipschitz. Let us fix a function $h \in L^{\infty}(\Lambda)$. For any $u \in L^{2}(\Lambda)$, we define
$$ 
(F(u))(x) = g \left( \| u \| \right) \cdot h(x), \quad x \in \Lambda, \, a.e.
$$
Similarly, one can consider examples like $(F(u))(x) =  h(x) \int_0^1 g(y) f(u(y)) \, dy$, $x \in \Lambda, \, a.e.$, where $g, h \in L^{\infty}(\Lambda)$ and $f: \R \rightarrow \R$ is Lipschitz and bounded.
\end{example}

\begin{example} \label{c44} This class  concerns with Nemytskii type nonlinearities. For any $u \in L^2(\Lambda)$, we define
\begin{equation} \label{gio1}
(F(u))(x) =  f(u(x)), \quad x \in \Lambda, \, a.e.,
\end{equation}
where  $f: \R \rightarrow \R$ is a $C^1$-function such that  the following conditions holds: There exists $C>0$, $q \ge 2$, $C_q >0$ such that
\begin{itemize}
\item[(i)] $|f(x)| \leq C (1 + |x|^2), \, x \in \R$,
\item[(ii)] $|f'(x)| \leq C (1 + |x|), \, x \in \R$, 
\item[(iii)] $f(x+y) \, y \leq C_q (1 + |y|^2 + |x|^q), \;\; \, x, y \in \R$.
\end{itemize}
For instance, we can fix $\eta \in [0,1]$ and consider $f(x) = - x |x|^{\eta} \, g(|x|^{1-\eta})$ where $g: \R \to \R$ is any bounded $C^1$-function which is Lipschitz and nonnegative (cf. Section \ref{subsection:examples} in Appendix).
\end{example}  

\begin{remark} \label{gy} Let us compare our examples with the  semilinear SPDEs considered in \cite{gyongy}. Indeed such paper also considers stochastic Burgers equations driven by space–time white noise  even in the multiplicative case. Our Examples \ref{11} and \ref{12} are not covered by \cite{gyongy}, which does not consider nonlocal non\-lin\-e\-ar\-i\-ties. Also it seems that Example \ref{c44} cannot be treated  in  \cite{gyongy}, since there for a nonlinearity like \eqref{gio1} the  author imposes a linear growth condition on $f$ (cf. hypothesis (2) on page 274 in \cite{gyongy}).
\end{remark}

\begin{remark} \label{gen} More generally, the previous examples can be formulated for nonlinearities involving $f: \bar{\Lambda} \times \R \rightarrow \R$. Such extensions are straightforward. We prefer to present the main examples for $f: \R \to \R$ for the sake of simplicity.

In Example \ref{c44} one could relax the conditions $f \in C^1(\R)$ and (ii) by requiring that $f: \R \rightarrow \R$ is locally Lipschitz with linearly growing Lipschitz constant (cf. hypothesis (3) on page 274 in \cite{gyongy}).
\end{remark}
 
\begin{remark} \label{open} 
Generalized stochastic Burgers equations can also be studied when  the nonlinear term $\partial_x \left( X^2(t) \right)$ is replaced by $\partial_x \left(g(X(t))\right)$, where $g$ grows more than quadratically. However it seems an open problem to study well-posedness and integrability properties of the solutions for  such equations  in the   cylindrical Wiener  case which we consider here (we refer to \cite{gyongyrovira}, \cite{kim} and \cite{kumar} for results about such general $g$ assuming a more regular noise).
\end{remark}

In this paper for the initial condition $X_0$ we mainly assume that it is deterministic and $X_0 \in H^{s} \cap C(\bar \Lambda)$, for any $s \in [0,1/2)$.
  
\vskip 1mm
The previous assumption on $X_0$ is used to a get a ``regular  solution'' of \eqref{eq:Burgers} (cf. Proposition \ref{prop:regulXtilde}). This is needed to obtain precise convergence results in the parameter estimation (see Section \ref{subsection:mainresults}). 

On the other hand, the well-posedness of our SPDE can be established more generally with $X_0 \in H$.

\vskip 1mm 
We can also consider a {\sl random initial condition $X_0$} that is $\mathcal F_0$-measurable with values in $L^4(\Lambda)$ such that, for any $\omega \in \Omega$, $\P$-a.s.,  $X_0(\omega) \in H^{s} \cap C(\bar \Lambda)$, for any $s \in [0,1/2)$. In addition we require that 
\begin{equation} \label{eq:ini}
X_0 \in L^k(\Omega; L^4(\Lambda)), \quad \text{for all $k \geq 1$,}
\end{equation}
in order to prove the integral bounds in Proposition \ref{prop:DPD}.

\vskip 1mm
By the {\sl mild solution to equation \eqref{eq:Burgers}}, we mean that for any $T>0$, there exists an $L^2(\Lambda)$-valued adapted process $X = (X(t), 0 \leq t \leq T)$ with continuous paths such that
\begin{equation} \label{mild solution}
X(t) = \underbrace{S_{\vartheta}(t) X_0 + \frac{1}{2} \int_0^t S_{\vartheta}(t-s) \partial_x \left( X^2(s) \right) \, ds +  \int_0^t S_{\vartheta}(t-s) F(X(s)) \, ds}_{=: \widetilde{X}(t)} + \underbrace{\int_0^t S_{\vartheta} (t-s) \, dW(s)}_{=: \bar{X}(t)},	
\end{equation}
for any $t \in [0, T]$, $\P$-a.s. In what follows, we use the ``splitting'' of the solution into a linear part $(\bar{X}(t), 0 \leq t \leq T)$ and a nonlinear part $(\widetilde{X}(t), 0 \leq t \leq T)$ as indicated. The $L^2(\Lambda)$-valued process $\bar{X} = (\bar{X}(t), 0 \leq t \leq T)$ is the so-called stochastic convolution and it is the unique mild solution to the following linear equation
\begin{equation} \label{eq:linear}
d\bar{X}(t) = \theta \Delta \bar{X}(t) \, dt + dW(t), \quad 0 < t \leq T, \quad \bar{X}(0) = 0. 
\end{equation}
(cf. Chapter 5 in \cite{DPZ}). The regularity we impose on $X_0$ is basically the spatial regularity of $\bar{X}(t)$ (see also Proposition \ref{prop:regulXbar}). 

On the other hand, the nonlinear part $\widetilde{X} = X - \bar{X}$ solves formally the PDE with random coefficients given by
\begin{equation} \label{eq:nonlinear}
\frac{d}{dt} \widetilde{X}(t) = \theta \Delta \widetilde{X}(t) + \frac{1}{2} \partial_x \big( (\bar{X}(t) + \widetilde{X}(t))^2 \big) + F \big(\bar{X}(t) + \widetilde{X}(t) \big), \quad 0 < t \leq T, \quad \widetilde{X}(0) = X_0,
\end{equation} 
\vskip -2mm 
\noindent
with zero Dirichlet boundary conditions $\widetilde{X}(t,0) = \widetilde{X}(t,1) = 0$, for $0 < t \leq T$. In the estimation procedure, we exploit the fact that $\widetilde{X}$ has higher regularity than $\bar{X}$. While the results on the regularity of $\bar{X}$ are basically known, the regularity of the nonlinear part $\widetilde{X}$ seems to be new in the present space-time white noise case even when $F=0$. We summarize them in the following proposition.

\begin{proposition} \label{prop:regulXtilde}
Let $X_0 \in  H^{s} \cap C(\bar \Lambda)$, for any $s \in [0, 1/2)$. Then there exists a pathwise unique mild solution $X$ to the equation \eqref{eq:Burgers} that satisfies $X \in C([0,T]; H^s) \cap C([0,T]; \bar C(\Lambda)) $, $\P$-a.s., for any $s \in [0, 1/2)$. 

The nonlinear part $\widetilde{X} \in  C([0,T]; H^s)\cap C([0,T]; \bar C(\Lambda)) $, $\P$-a.s., for any $s \in [0, 1/2)$. Moreover, for any $s \in [1/2, 3/2)$ with $\eps = \eps(s) \in (0, 1/2)$  such that $s - 1 + 2 \eps < 1/2$, and with the function $g_s (t) = t^{- 1/2 + \eps(s)}$, we have, $\P$-a.s.,
\begin{align} 
(i) \, \, \widetilde{X} \in C((0, T]; H^s), \label{continua} \quad
(ii) \, \, \|\widetilde{X}(t) \|_s \leq C_{\theta, s, \eps, X_0} \, g_s(t), \quad t \in (0,T],
\end{align}
for any $s \in [1/2, 3/2)$ and some (random) constant $C_{\theta, s, \eps, X_0} > 0$.
\end{proposition}
In the sequel, we will frequently use that the function $g_s \in L^2([0, T])$ for $s \in [1/2, 3/2)$. 
 
\vskip 1mm
The proof is deferred to Section \ref{sec:proofs}.
We also need to consider the notion of weak solution. In that regard, we formulate the following lemma.
\begin{lemma} \label{lemma:weak solution}
The mild solution $X$ verifies, for any $z \in H^2$, $\P$-a.s., for any $t \in [0,T]$,
$$
\left\langle z, X(t) \right\rangle = \left\langle z, X_0 \right\rangle + \theta \int_0^t \left\langle \Delta z, X(s) \right\rangle \, ds - \frac{1}{2} \int_0^t \left\langle \partial_x z, X^2(s) \right\rangle \, ds + \int_0^t \left\langle z, F(X(s)) \right\rangle \, ds + \left\langle z, W(t) \right\rangle.
$$
\end{lemma}
The proof is not difficult, but rather lengthy. We give a~sketch of proof in Section \ref{subsection:weak solution} in Appendix. 

\vskip 1mm
We finish the section with a result on well-posedness of \eqref{mild solution}. It holds when $X_0 \in H$. We stress that well-posedness of mild solutions for the Burgers equation \eqref{mild solution} with $F = 0$ is treated differently in the papers \cite{DPDT} (see also \cite{DPZ96}) and \cite{gatarekDap}; actually \cite{gatarekDap} also considers a multiplicative noise case. Below we show that the approach in \cite{gatarekDap} can be  adapted to prove well-posedness of \eqref{mild solution} under our assumptions.  
 
\begin{proposition} \label{gatarek}
Let $X_0 \in  H$. Then there exists a pathwise unique mild solution $X$ to the equation \eqref{eq:Burgers}.
\end{proposition}
\begin{proof}
Let $T > 0$. For any $n \geq 1$, we introduce the projection $\pi_n: H=L^2(\Lambda) \to B(0,n)$, with $B(0,n) = \{ u \in H; \, \| u \| = \| u \|_H \leq n\}$ (here $\pi_n(u) = u$ if $u \in B(0,n)$ and $\pi_n(u)=\frac{nu}{\| u \|}$ if $\| u \| > n$). We consider 
\begin{equation} \label{g12}
X_n(t) = {S_{\theta}(t) X_0 + \frac{1}{2} \int_0^t S_{\theta}(t-s) \partial_x \left( [\pi_n  (X_n(s)) ]^2 \right) \, ds} + \int_0^t S_{\theta}(t-s) F(\pi_n (X_n(s))) \, ds + \bar{X}(t), 
\end{equation} 
for $t \in [0,T]$, $\P$-a.s. Arguing as in \cite{gatarekDap} one can prove that \eqref{g12} has a unique solution on a~small time interval $[0, R] $, $0<R \le T,$ by the contraction principle, using the space $Z_{R}$; $Z_{R}$ is the Banach space of all continuous adapted $H$-valued  processes $u$ such that $\| u \|_{Z_{R}}^2 = \E [\sup_{t \in [0,{R}]} \| u(t)\|^2] < \infty$.
  
We stress that $R = R(n,T)$ is a {\it deterministic} constant (when $F=0$  in \cite{gatarekDap} one gets $R \le C n^{-4}$ where $C$ is deterministic). After a finite numbers of steps one finds that the solution $X_n$ exists and it is unique on $[0,T]$. Clearly, this solution is an adapted stochastic process.
   
We only show how to  handle the additional term $\int_0^t S_{\theta}(t-s) F(\pi_n (X_n(s))) \, ds$ which is not present in \cite{gatarekDap}. Let $0<R \le T$. For $u \in Z_R$, define the nonlinear operator $M_n: Z_R \rightarrow Z_R$ by
$$
(M_n u)(t) := \int_0^t S_{\theta}(t-s) F\left( \pi_n u(s) \right) \, ds, \quad t \in [0, R].
$$
\vskip -2mm 
\noindent
Let $u, v \in Z_R$ be such that $u(0) = v(0) = X_0$. We have, $\P$-a.s.,
\begin{align*} 
&\left\| (M_n u)(t) - (M_n v)(t) \right\| \leq \int_0^t \left\| S_{\theta}(t-s) \left( F(\pi_n (u(s))) - F(\pi_n (v(s))) \right) \right\| \, ds \\
&\lesssim \int_0^t \frac{1}{(t-s)^{1/4}} \left\|  F(\pi_n (u(s)) - F(\pi_n (v(s))) \right\|_{L^1(\Lambda)} \, ds \leq C_n \int_0^t \frac{1}{(t-s)^{1/4}} \left\| \pi_n (u(s)) - \pi_n (v(s)) \right\| \, ds \\
&\leq C_n \int_0^t \frac{1}{(t-s)^{1/4}} \left\| u(s) - v(s) \right\| \, ds \leq \frac{4}{3} C_n \, R^{3/4} \sup_{t \in [0,R]} \| u(t) - v(t) \|,  
\end{align*}  
where we have used the local Lipschitz property of $F$, since $\pi_n (u(s)), \pi_n (v(s)) \in B(0,n)$, and the inequality  
$$ 
\| S_{\theta}(t) \phi\| \leq \frac{C}{t^{1/4}} \| \phi\|_{L^1(\Lambda)}, \quad \phi \in L^1(\Lambda), \quad t \in (0,T],
$$
that holds for some constant $C= C(\theta) > 0$ (the proof is standard; it involves  the explicit representation of the heat kernel). We find  
$$
\| M_n u - M_n v \|_{Z_R} \leq \frac{1}{6} \| u - v \|_{Z_R},
$$ 
for small enough deterministic $R = R(n,T)$ and we can continue as in \cite{gatarekDap}. Indeed we can repeatedly apply the contraction principle on a finite number of closed time intervals $I_{R,k} \subset [0,T]$, $k = 1, \ldots, N = N(n)$, with length less or equal than $R$ (clearly, $\cup_{k=1}^N I_{R,k} = [0,T]$). Thus on each $I_{R,k}$ we have that $X_n$ can be obtained, $\P$-a.s., as limit on $C(I_{R,k}; H)$ of an approximating sequence of adapted stochastic processes.
    
Solutions of \eqref{g12} can be defined for all $t \geq 0.$ In \cite{gatarekDap} they also define the  stopping times:
\begin{equation} \label{tau}
\tau_n = \inf\{t \geq 0; \, \| X_n(t)\| \geq n \}, \quad n \geq 1.
\end{equation}
Since $X_n(t) = X_m(t)$, $m \geq n$, $t \leq \tau_n$ one can set $X(t) = X_n(t)$, $t \leq \tau_n$ and define a solution of \eqref{g12} on $[0, \tau_{\infty})$.
One can obtain a (global) mild solution by proving that  $\tau_{\infty}=\infty$, i.e.,  $\tau_n \rightarrow \infty$, $\P$-a.s by means of a-priori estimates as in  Lemma 3.1 of \cite{DPDT} (see also page 263 in \cite{DPZ96}). Let us explain this method to perform such estimates for \eqref{g12}.

Following \cite{DPZ96} we first consider a regular approximation of $\bar X(t)$. Note that $\bar{X} \in C([0,T]; C(\bar \Lambda))$, $\P$-a.s. (cf. Section 5 in \cite{DPZ}). With such approximation we can work with equation \eqref{eq:nonlinear} which holds for a.e. $t \in (0,T].$ Note that the final obtained estimates do not require the presence of the regular approximation of $\bar X(t)$; they hold with $\bar X(t)$. 
   
Multiplying both sides of equation \eqref{eq:nonlinear} by $\widetilde{X}$, integrating with respect to $x$ in $\Lambda = (0,1)$ and using the integration by parts we get
\begin{align}
&\frac{1}{2} \frac{d}{dt} \int_0^1 \widetilde{X}^2(t, x) \, dx + \theta \int_0^1 (\partial_x \widetilde{X}(t, x))^2 \, dx \notag \\
&= - \frac{1}{2} \int_0^1 \left( \widetilde{X}(t, x) + \bar{X}(t, x) \right)^2 \partial_x \widetilde{X}(t, x) \, dx + \int_0^1 F \left(\bar{X}(t, \cdot) + \widetilde{X}(t, \cdot) \right)(x)\, \widetilde{X}(t, x) \, dx \notag \\
&= - \frac{1}{2} \int_0^1 \widetilde{X}^2(t, x) \partial_x \widetilde{X}(t, x) \, dx - \int_0^1 \bar{X}(t, x) \widetilde{X}(t, x) \partial_x \widetilde{X}(t, x) \, dx \notag \\
&\phantom{=}- \frac{1}{2} \int_0^1 \bar{X}^2(t, x) \partial_x \widetilde{X}(t, x) \, dx + \int_0^1 F \left(\bar{X}(t, \cdot) + \widetilde{X}(t, \cdot) \right)(x) \,  \widetilde{X}(t, x) \, dx. \label{eq:integrability}
\end{align}
\vskip -1mm 
\noindent
The first term on the right-hand side equals to zero due to the boundary conditions
$$
\int_0^1 \widetilde{X}^2(t, x) \partial_x \widetilde{X}(t, x) \, dx = \frac{1}{3} \left[ \widetilde{X}^3(t, x) \right]_{x=0}^{x=1} = 0.
$$
For the last term on the right-hand side of \eqref{eq:integrability}, we use the assumption (c) in \eqref{bb} 
\begin{align} 
\int_0^1 F \left(\bar{X}(t, \cdot) + \widetilde{X}(t, \cdot) \right)(x)\,  \widetilde{X}(t, x) \, dx &\leq C_q \int_0^1 \left(1 + |\bar{X}(t, x)|^q + |\widetilde{X}(t, x)|^2 \right) \, dx \notag \\
&= C_q \left( 1 + \| \bar{X}(t) \|_{L^q(\Lambda)}^q + \| \widetilde{X}(t) \|^2 \right), \label{d66}
\end{align}
so we obtain 
\begin{align} \label{s66}
&\frac{1}{2} \frac{d}{dt} \| \widetilde{X}(t) \|^2 + \theta \| \widetilde{X}(t) \|_1^2 \\ \notag 
&\leq \| \bar{X}(t) \|_{C(\bar \Lambda)} \| \widetilde{X}(t) \|_{1} \| \widetilde{X}(t) \|_{} + \frac{1}{2} \| \bar{X}(t) \|_{C(\bar \Lambda)}^2 \| \widetilde{X}(t) \|_{1}  + C_q \left( 1 + \| \bar{X}(t) \|_{C(\bar \Lambda)}^q +  \| \widetilde{X}(t) \|^2 \right) 
\\ \notag  
&\leq \frac{\theta}{2} \| \widetilde{X}(t) \|_1^2 + \frac{2}{\theta} \| \bar{X}(t) \|_{C(\bar \Lambda)}^2 \| \widetilde{X}(t) \|_{}^2 +    \frac{\theta}{4} \| \widetilde{X}(t) \|_1^2 + \frac{1}{\theta} \| \bar{X}(t) \|_{C(\bar \Lambda)}^4  + C_q + C_q \| \bar{X}(t) \|_{C(\bar \Lambda)}^q + C_q \| \widetilde{X}(t) \|^2.  
\end{align}   
Therefore, we can easily apply the Gronwall lemma  as in the proof of Lemma 3.1 in  \cite{DPDT} and get  
\begin{equation}\label{sty} 
\| \widetilde{X}(t) \|^2 \le  C ( \mu + \|X_0\|^2  +1)e^{T \mu}
\end{equation}
with $t \in [0,T]$, $C = C(\omega, q,T, \theta) > 0$, $\mu = \sup_{t \in [0,T]} \| \bar X(t)\|^{q + 4}_{C (\bar \Lambda)}$. This finishes the proof.
\end{proof}

\subsection{The observation scheme}

As motivated in \cite{ACP}, \cite{AR}, \cite{JR}, we observe the solution process $(X(t,x), \, t \in [0,T], x \in \Lambda)$ only locally in space around some point $x_0\in\Lambda$. That point $x_0$ as well as the terminal time $T > 0$ remain fixed.

More precisely, the observations are given by a spatial convolution of the solution process with a~kernel $K_{\delta,x_0}$, localising at $x_0$ as the resolution $\delta$ tends to zero. This kernel might for instance model the {\it point spread function} in microscopy.

For $z \in L^2(\mathbb R)$ and $\delta \in(0,1]$ introduce the scalings
\begin{align*}
\Lambda_{\delta, x_0} &:= \delta^{-1} \left( \Lambda - x_0 \right) = \{ \delta^{-1} (x - x_0); \, x \in \Lambda \}, \\
z_{\delta, x_0}(x) &:= \delta^{-1/2} z\left( \delta^{-1} (x - x_0) \right), \quad x \in \R,
\end{align*}
and we also set $\Lambda_{0,x_0} := \R$. For $\delta \in (0, 1]$, we denote by $\Delta_{\delta, x_0}$ the Laplace operator on $L^2(\Lambda_{\delta, x_0})$ with Dirichlet boundary conditions.
  
Throughout this paper, $K \in H^2(\mathbb R)$ stands for a fixed function of compact support in $\Lambda_{\delta', x_0}$, for some $0 < \delta' \leq 1$, called kernel (here $H^2(\mathbb R)$ is the usual $L^2$-Sobolev space on $\R$); note that $\delta' $ may depend on $x_0$. The compact support ensures that $K_{\delta, x_0}$ is localising around $x_0$ as $\delta\to 0$; moreover we have $K_{\delta, x_0} \in H^2=  H^2(\Lambda) $, $0< \delta \le \delta'$. The scaling with $\delta^{-1/2}$ simplifies calculations due to $\| K_{\delta, x_0} \| = \| K \|_{L^2(\mathbb R)}$, while the proposed estimator is invariant with respect to kernel scaling.
  
Local measurements of \eqref{eq:Burgers} at the point $x_0$ with resolution level $\delta\in(0,1]$ are described by the real-valued processes $(X_{\delta, x_0}(t), 0 \leq t \leq T)$ and $(X_{\delta, x_0}^{\Delta}(t), 0 \leq t \leq T)$ given by
\begin{align}
X_{\delta, x_0}(t) &= \left\langle X(t), K_{\delta, x_0} \right\rangle, \label{X-delta} \\
X_{\delta, x_0}^{\Delta}(t) &= \left\langle X(t), \Delta K_{\delta, x_0} \label{X-delta-delta} \right\rangle.
\end{align}
\vskip -2mm 
\noindent
These two processes are the data for our estimation procedure. In fact, since $X_{\delta, x_0}^{\Delta}(t) = \Delta X_{\delta, \cdot}(t)|_{x=x_0}$ by convolution, $X_{\delta, x_0}^{\Delta}(t)$ can be computed by observing $X_{\delta, x}(t)$ for $x$ in a neighborhood of $x_0$.

\section{Estimation method and main results} \label{sec:mainresults}

\subsection{The estimator}

We use the augmented maximum likelihood estimator $\hat{\theta}_{\delta}$ of $\theta$ introduced in \cite{AR}. This (nonparametric) estimator was derived for a linear stochastic heat equation with additive space-time white noise, but \cite{ACP} studied it in an abstract nonlinear (and parametric) setting that also covers the stochastic Burgers equation with trace class noise.

\begin{definition}
The {\it augmented maximum likelihood estimator} (augmented \hbox{MLE}) $\hat\vartheta_{\delta}$ of the parameter $\vartheta > 0$ is defined as
\begin{equation} \label{eq:estimator}
\hat{\theta}_{\delta} = \frac{\int_0^T X_{\delta, x_0}^{\Delta}(t) \, dX_{\delta, x_0}(t)}{\int_0^T \left( X_{\delta, x_0}^{\Delta}(t) \right)^2 \, dt}.
\end{equation}
\end{definition}

As discussed in \cite{AR}, this estimator is closely related to, but different from, the actual MLE, which cannot be computed in closed form, even for linear equations and constant $\theta$.

By Lemma  \ref{lemma:weak solution} the dynamics of $X_{\delta, x_0}$ is given by
\begin{equation} \label{eq:scalardynamics}
dX_{\delta, x_0} = \vartheta X_{\delta, x_0}^{\Delta} \, dt + \frac{1}{2} \left\langle \partial_x \left( X^2(t) \right), K_{\delta, x_0} \right\rangle \, dt + \left\langle F(X(t)), K_{\delta, x_0} \right\rangle \, dt + \| K \|_{L^2(\R)} \, d\bar{w}(t),
\end{equation} 
where $\bar{w}(t) := \left\langle W(t), K_{\delta, x_0} \right\rangle / \| K_{\delta, x_0} \|$ is a scalar Brownian motion. (Note that $\| K_{\delta, x_0} \| = \| K \|_{L^2(\R)} > 0$). Using \eqref{eq:scalardynamics} in the numerator on the right-hand side of \eqref{eq:estimator}, we obtain the fundamental error decomposition
\begin{equation} \label{eq:error}
\delta^{-1} (\hat{\vartheta}_{\delta} - \vartheta) = \delta^{-1} (\mathcal I_{\delta})^{-1} \mathcal R_{\delta} + \delta^{-1} (\mathcal I_{\delta})^{-1} \mathcal M_{\delta},
\end{equation}
where
\begin{align*}
\mathcal I_{\delta} &:= \| K \|_{L^2(\R)}^{-2} \int_0^T \left( X_{\delta, x_0}^{\Delta}(t) \right)^2 \, dt, \tag{observed Fisher information} \\
\mathcal R_{\delta} &:= \| K \|_{L^2(\R)}^{-2} \frac{1}{2} \int_0^T X_{\delta, x_0}^{\Delta}(t) \left\langle \partial_x \left( X^2(t) \right), K_{\delta, x_0} \right\rangle \, dt \\
&\phantom{=}+ \| K \|_{L^2(\R)}^{-2} \int_0^T X_{\delta, x_0}^{\Delta}(t) \left\langle F(X(t)), K_{\delta, x_0} \right\rangle \, dt, \tag{nonlinear bias} \\
\mathcal M_{\delta} &:= \| K \|_{L^2(\R)}^{-1} \int_0^T X_{\delta, x_0}^{\Delta}(t) \, d\bar{w}(t). \tag{martingale part}
\end{align*}
The observed Fisher information $\mathcal I_{\delta}$ does not correspond to the Fisher information of the statistical model, however it gives the quadratic variation of the martingale $\mathcal M_{\delta}$ in time.

To prove consistency, it is enough to show that $(\mathcal I_{\delta})^{-1} \mathcal R_{\delta}$ and $ (\mathcal I_{\delta})^{-1} \mathcal M_{\delta}$ vanish, as $\delta \rightarrow 0$, and to prove asymptotic normality, we will show that $\delta^{-1} (\mathcal I_{\delta})^{-1} \mathcal R_{\delta} \rightarrow 0$, while $\delta^{-1} (\mathcal I_{\delta})^{-1} \mathcal M_{\delta}$ converges in distribution to a Gaussian random variable.

To analyze these terms, we use the ``splitting technique'' of the solution (i.e., we write $X = \bar{X} + \widetilde{X}$ as above) and we study separately the linear parts
\begin{align*}
\bar{X}_{\delta, x_0}(t) := \left\langle \bar{X}(t), K_{\delta, x_0} \right\rangle, \;\;\;\;   \;\;\;
\bar{X}_{\delta, x_0}^{\Delta}(t) := \left\langle \bar{X}(t), \Delta K_{\delta, x_0} \right\rangle  
\end{align*}
and the corresponding nonlinear parts $\widetilde{X}_{\delta, x_0}(t)$, $\widetilde{X}_{\delta, x_0}^{\Delta}(t)$. (Note that since the stochastic convolution $\bar{X}$ is centered Gaussian process, the $\bar{X}_{\delta, x_0}$ and $\bar{X}_{\delta, x_0}^{\Delta}$ are also centered Gaussian processes.) We also define the observed Fisher information $\bar{\mathcal I}_{\delta}$ that corresponds to the linear part
$$
\bar{\mathcal I}_{\delta} := \| K \|_{L^2(\R)}^{-2} \int_0^T \left( \bar{X}_{\delta, x_0}^{\Delta}(t) \right)^2 \, dt.
$$
For the analysis, we initially follow \cite{ACP}, where the classical stochastic Burgers equation with spatially ``smooth\-ed'' noise $(- \Delta)^{-\gamma} \, dW(t)$, $\gamma > 1/4$ is also discussed. (It is a~borderline case of the presented more general nonlinear SPDE and as that, it is analyzed separately.) A lot of general and auxiliary results can be used also in our case $\gamma = 0$, so first, let us revisit the assumptions of the model from \cite{ACP}: 
 
\vskip 2mm 
{\it Assumption $B$:} Since we have $\gamma = 0$, $B = I$, $B^* = I$, $B_{\delta, x_0}^* = I$, this assumption trivially fulfilled.

{\it Assumption $K$:} Since we have $\gamma = 0$, $\lceil {\gamma} \rceil = 0$, this assumption is fulfilled by taking $\widetilde{K} := K$.

{\it Assumption $ND$:} Since we have $\gamma = 0$, $\lceil {\gamma} \rceil = 0$, $B_{\delta, x_0}^* = I$, it is required that $\| K \|_{L^2(\R)} > 0$ and $\frac{1}{2} \| K' \|_{L^2(\R)}^2 > 0$. We assume $K$ with a compact support, so this non-degeneracy condition is fulfilled.
 
{\it Assumption $F$:} Since our nonlinearity $\frac{1}{2} \partial_x \left( X^2(\cdot) \right) + F(X(\cdot))$ is more concrete than rather general nonlinearity $F$ considered in \cite{ACP}, we do not need $(16)$ from \cite{ACP} (that is used to claim (iv) in Proposition 3 therein). We only need the part $(15)$. 

By Lemma \eqref{lemma:upperbounds}(i) in Appendix we show that the inequality $\left| \widetilde{X}_{\delta}^{\Delta}(t) \right| \lesssim g_{3/2 - \eps}(t) 
\delta^{-1/2 - \eps}$ holds for any $\eps \in [0,1/2)$, $\P$-a.s., $t \in (0,T]$ (recall that $g_{3/2 - \eps} \in L^2([0, T])$, see also \eqref{continua}). That provides 
$$ 
\int_0^T \left( \widetilde{X}_{\delta, x_0}^{\Delta}(t) \right)^2 \, dt = o_{\P}(\delta^{-2}),
$$
which is enough to obtain the point (iii) in Proposition \ref{prop:asymptoticsbarI}. (Cf. Proposition 3(iii) in \cite{ACP}.) 

\vskip 1mm
With this, the only assumption from \cite{ACP} that we pose on our model is on the kernel $K$ and comes from Theorem 13 in \cite{ACP}. Recall that $x_0 \in \Lambda = (0,1)$ is fixed.

\begin{assumption}{(L)}\label{ass:L}
There exists a function $L \in H^3(\R)$ with compact support in $\Lambda_{\delta', x_0}$, for some $0 < \delta' \leq 1$, such that $K = \partial_x L  = L'$. 
\end{assumption} 
   
\begin{example} We provide an easy example of kernel $K$ which verifies Assumption \ref{ass:L}. Consider \hbox{a~smooth} compactly supported bump function
$$
f(x) := \exp \left( - \frac{10}{1 - x^2} \right) \boldsymbol{1}_{[-1,1]}(x), \quad x \in \R. 
$$
Since $[-1, 1] \subset \Lambda_{\delta', x_0} = (- \delta'^{-1} x_0, \delta'^{-1} (1 - x_0))$ for some $0 < \delta' \leq 1$ (depending on $x_0$)  and $f \in H^3(\R)$, we can take $L := f$ and the kernel $K$ for the statistical procedure as $K := f'$. Assumption \ref{ass:L} is satisfied. Note though, that different choices and setups are also possible.
\end{example}
 
In order to prove the asymptotic normality of our estimator, we follow \cite{ACP} in the analysis of the martingale term $\delta^{-1} (\mathcal I_{\delta})^{-1} \mathcal M_{\delta}$ as $\delta \rightarrow 0$, but we deviate for the analysis of the bias term $\delta^{-1} (\mathcal I_{\delta})^{-1} \mathcal R_{\delta}$ mainly for two reasons. First, we need to adjust Lemma S.3 of the supplement of \cite{ACP}  to our case $\gamma = 0$, $p = 2$ for $\widetilde{X}$ and the regularity of $\bar{X}$ and $\widetilde{X}$ (for the comparison, see Lemma \ref{lemma:upperbounds} in Appendix). Second, we want to clarify some points in the paper (see the remark below).

\begin{remark}\label{altri}
From the calculations in Lemma S.4 of the supplement of \cite{ACP}, it seems there should be a factor $\lambda_k^{- \gamma - 1/2}$ (instead of just $\lambda_k^{- \gamma}$) in (S.4) and this is then too big for the sum in (S.3) to converge. (Such result comes from Lemma S.2(ii) and it is not quite clear if Lemma S.1 is able to tame it even in the case $\gamma > 1/4$.) Moreover, in the proof of Lemma S.8, it is not completely  clear why the random variable $\widetilde{X}(t,0)$ (that is basically $\widetilde{X}(t, x_0)$) is $\mathcal G$-measurable. The isonormal Gaussian process $\widetilde{W}(z)$ is defined such that it does not correspond with the norm $\| \cdot \|_{\mathcal H}$ of the presented Hilbert space $\mathcal H$. (Basically, some sort of integration with respect to $x$ is missing.) However, if the definition of $\widetilde{W}(z)$ is changed, the $\sigma$-field $\mathcal G$ is changed and then one should verify that the random variable $\widetilde{X}(t,0)$ is $\mathcal G$-measurable. 
\end{remark}

We therefore develop new arguments that can overcome these difficulties. Our presented proofs of Lemmas \ref{lemma:U2} and \ref{lemma:U3} rely on a different technique and cover not only our case $\gamma = 0$. We believe that with some extra work they could be also adjusted to cover the case of any $\gamma \geq 0$. 

Finally, when all convergences in \eqref{eq:error} are assembled, we can formulate our main result in the following theorem. The proof is deferred to Section \ref{sec:proofs}.
\begin{theorem} \label{thm:main}
Grant Assumption \ref{ass:L}. Then the augmented maximum likelihood estimator $\hat{\vartheta}_{\delta}$ is strongly consistent and asymptotically normal estimator of the parameter $\vartheta$ satisfying
\begin{equation} \label{eq:mainresult}
\delta^{-1} \left( \hat{\theta}_{\delta} - \theta \right) \stackrel{d}{\rightarrow} N \left( 0, \frac{2 \theta \| K \|_{L^2(\R)}^2}{T \| K' \|_{L^2(\R)}^2} \right), \quad \text{as $\delta \rightarrow 0$.}
\end{equation}  
\end{theorem}
\vskip -1mm 
\noindent
The asymptotic normality of the estimator allows us to prescribe asymptotic confidence intervals for the parameter $\theta$.
\begin{corollary}
Let $\alpha \in (0,1)$. Based on the estimator $\hat{\theta}_{\delta}$, the confidence interval
$$
I_{1 - \alpha} = \left[ \hat{\theta}_{\delta} - \mathcal I_{\delta}^{-1/2} q_{1 - \alpha/2}, \hat{\theta}_{\delta} + \mathcal I_{\delta}^{-1/2} q_{1 - \alpha/2} \right],
$$
with the standard Gaussian $(1 - \alpha/2)$-quantile $q_{1 - \alpha/2}$ has asymptotic coverage $1 - \alpha$ as $\delta \rightarrow 0$ under the assumptions of Theorem \ref{thm:main}.
\end{corollary}

The next result is adapted from  Theorem 6 in \cite{ACP}; see also the proof of Proposition 5.12 in \cite{AR}. It shows that the convergence rate $\delta$ is minimax optimal. We provide the proof in Section \ref{subsection:rate optimality} in Appendix.

Let $\Theta$ be the set of all admissible inputs $\kappa = (\theta, a, F,  X_0)$ in \eqref{eq:Burgers}, where we consider $\theta > 0$, $a \, \partial_x (X^2(\cdot))$, with $a \in \R$, instead of $\frac{1}{2} \partial_x (X^2(\cdot))$, the nonlinerarity $F$ that fulfills hypothesis \eqref{bb} and the initial random condition $X_0$ which verifies the assumptions corresponding to  \eqref{eq:ini}.

Denote by $\P_{\kappa}$ the law of $X_{\delta, x_0}$ on the canonical space $C([0,T])$, equipped with the Borel sigma-field corresponding to the supremum norm on $[0,T]$, and by $\E_{\kappa}$ its expectation.
\vskip -2mm 
\noindent
\begin{proposition} \label{prop:lowerbound}
Grant Assumption \ref{ass:L}. As $\delta \rightarrow 0$, we have the following asymptotic lower bound of the root mean squared error
\begin{equation*}
\inf_{\hat{\theta}} \sup_{\kappa \in \Theta} \left( \E_{\kappa} \left( \hat{\theta} - \theta \right)^2 \right)^{1/2} \geq C T^{-1/2} \delta,
\end{equation*}
for some constant $C > 0$, where the infimum is taken over all  estimators $\hat{\theta}$ based on observing $X_{\delta, x_0}$.
\end{proposition} 

Our overall results are satisfactory, because they match not only to the results in \cite{ACP}, but also to \cite{AR}, \cite{JR} and other works that studied the asymptotic behaviour of the augmented MLE in the framework of local measurements.

\section{Proofs} \label{sec:proofs}

\subsection{Moment estimates of the solution}

We prove an integral estimate for the $L^4(\Lambda)$-norm of the solution to generalized stochastic Burgers equation \eqref{eq:Burgers}. We adjust Proposition 2.2 in \cite{DPD} to our situation. This result holds for $X_0 \in L^4(\Lambda)$.

\begin{proposition} \label{prop:DPD}
For all $k \geq 1$ there exist $c_1(k)$, $c_2(k)$ such that for all $X_0 \in L^4(\Lambda)$ and all $T > 0$ we have
\begin{equation}
\E \Big( \sup_{t \in [0,T]} \| X(t) \|_{L^4(\Lambda)}^k \Big) \leq c_1(k) e^{c_2(k)T} \left( 1 + \| X_0 \|_{L^4(\Lambda)}^k \right).
\end{equation} 
\end{proposition}
\begin{proof}
We follow the setting, notation and the proof of Proposition 2.2 in \cite{DPD}. Without loss of ge\-ne\-ra\-li\-ty, we assume that $\theta = 1$. For any $\alpha \geq 0$, we set
$$
z_{\alpha}(t) = \int_0^t e^{(t-s)(\Delta - \alpha)} \, dW(s), \quad Y(t) = X(t) - z_{\alpha}(t), \quad 0 \leq t \leq T,
$$
where $(X(t), 0 \leq t \leq T)$ is the unique solution to the equation \eqref{eq:Burgers} (see Proposition \ref{gatarek}). Note that  $(Y(t), 0 \leq t \leq T)$ fulfills
$$
\frac{d}{dt} Y(t) = \partial_{xx}^2 Y(t) + \partial_x (Y(t) + z_{\alpha}(t))^2 + \alpha z_{\alpha}(t) + F(Y(t) + z_{\alpha}(t)).
$$
We follow the proof of Proposition 2.2 in \cite{DPD} using $p=4$ therein. We multiply the above equation by $4|Y(t)|^2 Y(t)$, integrate in $\bar \Lambda = [0,1]$ and then we integrate by parts. We receive the same terms $I_1$, $I_2$ and $I_3$ as in \cite{DPD} and we handle them in the same way. Here we handle only the additional term $I_4$: 
$$
I_4 = 4 \int_0^1 |Y(t, x)|^2 Y(t, x) \, F(Y(t, \cdot) + z_{\alpha}(t, \cdot))(x) \, dx.
$$

Using the assumption (c) in \eqref{bb} on the nonlinearity $F$ and the Young inequality, we have
\begin{align*}
I_4 &\leq 4 C_q \int_0^1 |Y(t, x)|^2 \left(1 + |z_{\alpha}(t, x)|^q + |Y(t, x)|^2 \right) \, dx \\
&\leq 4 C_q \int_0^1 |Y(t, x)|^2 \, dx + 4 C_q \int_0^1 |Y(t, x)|^2 |z_{\alpha}(t, x)|^q \, dx + 4 C_q \int_0^1 |Y(t, x)|^4 \, dx \\
&\leq 4 C_q \| Y(t) \|^2 + 6 C_q \| Y(t) \|_{L^4(\Lambda)}^4 + 2 C_q \| z_{\alpha}(t) \|_{L^{2q}(\Lambda)}^{2q} 
\leq C_1 \left( 1 + \| Y(t) \|_{L^4(\Lambda)}^4 + \| z_{\alpha}(t) \|_{L^{2q}(\Lambda)}^{2q} \right),
\end{align*}
\vskip -2mm 
\noindent
for some constant $C_1 > 0$. Therefore we obtain an upper bound that is analogous to the equation after (2.11) in \cite{DPD}
\begin{align*}
\frac{d}{dt} \| Y(t) \|_{L^4(\Lambda)}^4 &\leq C_2 + C_2 \left( \| z_{\alpha}(t) \|_{L^4(\Lambda)}^{8/3} + \| z_{\alpha}(t) \|_{L^4(\Lambda)} + 1 \right) \| Y(t) \|_{L^4(\Lambda)}^4 \\
&\phantom{=}+ C_2 \left( \alpha^4 \| z_{\alpha}(t) \|_{L^4(\Lambda)} + \| z_{\alpha}(t) \|_{L^4(\Lambda)}^8 + \| z_{\alpha}(t) \|_{L^{2q}(\Lambda)}^{2q} \right), 
\end{align*}
\vskip -2mm 
\noindent
for some constant $C_2 > 0$. Using Proposition 2.1 in \cite{DPD}, we choose $\alpha$ such that
$$
\sup_{t \in [0,T]} \| z_{\alpha}(t) \|_{L^{2q}(\Lambda)} \leq 1.
$$
Since $q \geq 2$, we also receive $\sup_{t \in [0,T]} \| z_{\alpha}(t) \|_{L^4(\Lambda)} \leq 1$ and we arrive at
$$
\frac{d}{dt} \| Y(t) \|_{L^4(\Lambda)}^4 \leq C_2 + 3C_2 \| Y(t) \|_{L^4(\Lambda)}^4 + C_2 (\alpha^4 + 2) \leq 2C_3 \| Y(t) \|_{L^4(\Lambda)}^4 + C_3 (\alpha^4 + 1),
$$
for some constant $C_3 > 0$. Now we use Gronwall lemma and the proof is finished as in \cite{DPD}.
\end{proof}

\subsection{Regularity of the linear part}
We consider path regularity of the stochastic convolution $\bar X$ (that is the linear part of the solution defined in \eqref{mild solution}).  

\begin{proposition} \label{prop:regulXbar}  
The process $\bar{X}$ has a continuous version with values in $H$. Moreover, for all $2 \leq p < \infty, \, 0 \leq  s < 1/2$: $\bar{X} \in C([0,T]; W^{s,p}(\Lambda))$, $\P$-a.s. In particular, we have  $\bar{X} \in C([0,T]; C(\bar \Lambda))$, $\P$-a.s. 
\end{proposition}    

For the proof we can follow Section 5 in \cite{DPZ} (using the so-called factorization method) or Proposition S.9 in \cite{ACP} with $\gamma = 0$, $d = 1$ and $B = I$. The last assertion follows by the well-known embedding: $W^{s,p}(\Lambda) \subset C^{0,s - 1/p}(\Lambda) \subset C(\bar \Lambda)$ which holds if in addition $p > 1/s$. 

By Proposition \ref{prop:algebra} we know  that, for any $s \in [0,1]$, we have, for $u, v \in H^s \cap L^{\infty}(\Lambda)$,
\begin{equation*} 
\| uv \|_s \leq \| u \|_{\infty} \| v \|_s + \| v \|_{\infty} \| u \|_s.
\end{equation*}
As a consequence we obtain the following additional re\-gu\-la\-ri\-ty result.

\begin{proposition} \label{prop:algebra2}
The process $\bar{X}$ defined in \eqref{mild solution} has the following property: For any $s \in [0, 1/2)$, we have
$
\bar{X}^2 \in C([0,T]; H^s), $ $ \P$-a.s. 
\end{proposition}   
\begin{proof} 
Fix $s \in [0, 1/2)$. By Proposition \ref{prop:algebra}, we write for any $t, r \in [0,T]$,  $\P$-a.s.,
\begin{align*}
\| \bar{X}^2(t) - \bar{X}^2(r) \|_s &= \| (\bar{X}(t) + \bar{X}(r)) (\bar{X}(t) - \bar{X}(r)) \|_s \\
&\leq \| \bar{X}(t) + \bar{X}(r) \|_{\infty} \| \bar{X}(t) - \bar{X}(r) \|_s + \| \bar{X}(t) - \bar{X}(r) \|_{\infty} \| \bar{X}(t) + \bar{X}(r) \|_s \\
& \leq C_T  \| \bar{X}(t) - \bar{X}(r) \|_s + C_T \| \bar{X}(t) - \bar{X}(r) \|_{\infty},
\end{align*} 
for some positive constant $C_T$ possibly depending on $\omega \in \Omega$. Passing to the limit $t \rightarrow r$ gives the assertion by Proposition \ref{prop:regulXbar}.
\end{proof}

\subsection{Regularity of the nonlinear part}

In this section we prove Proposition \ref{prop:regulXtilde}. We follow the ideas from Section 7.1 in \cite{P} and, in particular, the proof of Proposition 24 in \cite{P}. For $g \in C([0,T]; H)$, we define the operator $S$ by
$$
(Sg)(t) := \int_0^t S_{\theta}(t-r) g(r) \, dr, \, t \in [0,T]
$$
and we recall the assertions (7.7) and (7.8) from \cite{P} in our notation (see in particular (i) and (ii) below). We also add the well-known assertion (iii).

\begin{lemma}\label{lemma:TS}
\hspace{2em}
\begin{itemize}
\item[(i)] $T := (-A)^{-1/2} \partial_x$ is a bounded linear operator from $H^s$ into $H^s$, $s \in [0,1]$.
\item[(ii)] $S$ is a bounded linear operator from $C([0,T]; H)$ into $C([0,T]; H^s)$, $s \in [0,2)$.
\item[(iii)] $S$ is a bounded linear operator from $C([0,T]; L^1(\Lambda))$ into $C([0,T]; H^s)$, $s \in [0,3/2)$.
\end{itemize}
\end{lemma}
Assertion (iii) can be easily deduced by the following well-known estimate: For any $s \in [0, 3/2)$, $f \in L^{1}(\Lambda)$, we have that $S_{\theta}(t) f \in H^s$, $t > 0$, and
$$
\| S_{\theta}(t) f \|_{H^s} \leq C_{s, \theta} \frac{1}{t^{s/2 + 1/4}} \| f \|_{L^1(\Lambda)}. 
$$
This follows using  the semigroup law, estimate in (2.6) in \cite{DPDT} and the fact that $\| S_{\theta}(t) f \|_H \leq C_{\theta} \frac{1}{t^{1/4}} \| f \|_{L^{1}(\Lambda)}$. Note that $s/2 + 1/4 < 1$.

For $u \in C([0,T]; H^1)$, define the operator $R$ by
\begin{equation} \label{operator R}
(Ru)(t) := \int_0^t S_{\vartheta}(t-r) \partial_x u(r) \, dr, \quad t \in [0,T].
\end{equation}
We summarize the properties of the operator $R$.

\begin{lemma}\label{lemma:R}
\hspace{2em}
\begin{itemize}
\item[(i)] The operator $R$ can be extended to a linear, bounded ope\-ra\-tor from $C([0,T]; L^1(\Lambda))$ into $C([0,T]; H^s)$, $s \in [0, 1/2)$.
\item[(ii)] The operator $R$ can be extended to a linear, bounded operator from $C([0,T]; H)$ into $C([0,T]; H^s)$, $s \in [0, 1)$.
\item[(iii)] The operator $R$ can be extended to a linear, bounded operator from $C([0,T]; H^s)$, $s \in [0, 1/2)$ into $C([0,T]; H^s), \, s \in [0, 3/2)$.
\item[(iv)] The mapping $u \mapsto R(u^2)$ is continuous from $C([0,T]; H)$ into $C([0,T]; H^s)$, $s \in [0, 1/2)$.
\end{itemize}
\end{lemma}
\begin{proof}
(i). This assertion follows from Lemma 14.2.1 of \cite{DPZ96}.

(ii). Using Lemma \ref{lemma:TS}(i) we can write, for $u \in C([0,T]; H)$,
$$
(Ru)(t) = \int_0^t S_{\vartheta}(t-r) \partial_x u(r) \, dr = \int_0^t S_{\vartheta}(t-r) (-A)^{1/2} [(-A)^{-1/2} \partial_x] u(r) \, dr, \quad t \in [0,T].
$$
Note that $[(-A)^{-1/2} \partial_x]u \in C([0,T]; H)$. By Lemma \ref{lemma:TS}(ii) we know that, for any $\eps \in (0,1]$,
$$
t \mapsto (-A)^{1 - \eps} \int_0^t S_{\vartheta}(t-r) [(-A)^{-1/2} \partial_x] u(r) \, dr
$$
belongs to $C([0,T]; H)$. Hence
$$
(-A)^s Ru \in C([0,T]; H), \, s \in [0, 1/2), \text{ i.e., } Ru \in C([0,T]; H^s), \, s \in [0, 1).
$$

(iii). Fix $s \in [0, 1/2)$ and take $u \in C([0,T]; H^s)$. We can write
\begin{align}
(Ru)(t) &= \int_0^t S_{\vartheta}(t-r) \partial_x u(r) \, dr = (-A)^{1/2} \int_0^t S_{\vartheta}(t-r) (-A)^{-s/2} \underbrace{(-A)^{s/2} [(-A)^{-1/2} \partial_x] u(r)}_{=: g(r)} \, dr \notag \\
&= (-A)^{1/2 - s/2} \int_0^t S_{\vartheta}(t-r) g(r) \, dr. \label{eq:eta}
\end{align}

By Lemma \ref{lemma:TS}(i), the function $g \in C([0,T]; H)$ and consequently the integral on the right-hand side of \eqref{eq:eta} belongs to $C([0,T]; H^{2 - \eps})$ for any small $\eps > 0$ by  Lemma \ref{lemma:TS}(ii). But that \hbox{means} that $Ru \in C([0,T]; H^{1 + s - \eps})$. Since we can take $s$ as close to $1/2$ as we wish, we conclude with $Ru \in C([0,T]; H^s)$ for $s \in [0, 3/2)$.

(iv). The mapping: $h \mapsto h^2$ is continuous from $C([0,T]; H)$ into $C([0,T]; L^1(\Lambda))$, so the assertion follows from (i). 
\end{proof}

For $u \in C([0,T]; H)$, define the operator $Q$ by
\begin{equation} \label{def:Q}
(Qu)(t) := \int_0^t S_{\vartheta}(t-r) F(u(r)) \, dr, \quad t \in [0,T].
\end{equation}
Note that for any $t \in [0,T]$, $u(t)$ belongs to a compact set in $H$, hence \eqref{def:Q} is well defined. We formulate the desired property of the operator $Q$ in the following lemma.

\begin{lemma} \label{lemma:Q}
For any $R > 0$, the operator $Q$ is a nonlinear operator from $C([0,T]; B(0,R))$ into $C([0,T]; H^s)$, $s \in [0,3/2)$.
\end{lemma}
\begin{proof}
Fix $R > 0$ and take $u \in C([0,T]; B(0,R))$. In such situation $F = F|_{B(0,R)}$ and we have $Qu = S(F(u(\cdot)))$. First, we will show that $F(u(\cdot)) \in C([0,T]; L^1(\Lambda))$.

Clearly, $F(u(r)) \in L^1(\Lambda)$ for any $r \in [0,T]$ and by the local Lipschitz property of $F$, we have
$$
\| F(u(r)) - F(u(s)) \|_{L^1(\Lambda)} \leq C_R \| u(r) - u(s) \| \rightarrow 0,
$$
as $r \rightarrow s$. Hence $F(u(\cdot)) \in C([0,T]; L^1(\Lambda))$ for $u \in C([0,T]; B(0,R))$ and the target space of the operator $Q$ is then provided by Lemma \ref{lemma:TS}(iii).
\end{proof}

\begin{proof}[Proof of Proposition \ref{prop:regulXtilde}]
Recall that for any $t \in [0,T]$
$$
X(t) = \widetilde{X}(t) + \bar{X}(t) = S_{\theta}(t) X_0 + \frac{1}{2} (R(X^2))(t) + (Q(X))(t) + \bar{X}(t).
$$
{\bf  Step 1}. We first establish  the spatial regularity of  $v_{\theta}(t) = S_{\theta}(t) X_0$. 

It is easy to see that  $v_{\theta} \in C([0,T]; H^s) \cap C([0,T]; \bar C(\Lambda))$, for any $s \in [0, 1/2)$, using all the regularity properties of $X_0$. Let us check that property \eqref{continua} holds when $\widetilde X$ is replaced by $v_{\theta}$. 

Let us fix $s \in [1/2, 3/2)$. We  consider $\eps \in (0,1/2)$ such that $s-1 + 2 \eps < 1/2$. It is clear that $v_{\theta} \in C((0,T]; H^s)$ using the regularizing properties of the heat semigroup. Moreover,  
\begin{align*}
\|t^{1/2 - \eps} (-\Delta)^{s/2} v_{\theta}(t)\| &= t^{1/2 - \eps} \| (-\Delta)^{1/2 - \eps} S_{\theta}(t) (-\Delta)^{s/2 -\frac{1}{2} + \eps} X_0 \| \\
&\leq C_{s, \eps} t^{1/2 - \eps} \frac{1}{ t^{1/2 - \eps} } \, \| X_0\|_{H^{s-1 + 2 \eps}} \leq C_{s, \eps} \| X_0\|_{H^{s-1 + 2 \eps}},
\end{align*}
for  $t \in (0,T]$ (see, for instance, formula (5.24) in page 135 of \cite{DPZ}). We have verified \eqref{continua} for $v_{\theta}$.   
  
\vskip 1mm
\noindent
{\bf  Step 2}. We recall that $X$, $\bar X$ and $\widetilde{X} \in C([0,T]; H^s)$, $\P$-a.s., for $s \in [0, 1/2)$.

We know that the solution $X \in C([0,T]; H^s)$, $\P$-a.s., for $s \in [0, 1/2)$ (cf. Lemma 14.2.1 in \cite{DPZ96}). For the stochastic convolution $\bar{X}$,  we know that from Proposition \ref{prop:regulXbar}, for the term $R(X^2)$ of the nonlinear part $\widetilde{X}$, it follows from Lemma \ref{lemma:R}(iv), and for the term $Q(X)$, it follows from Lemma \ref{lemma:Q}. (We fix $\omega \in \Omega'$ for some event which verifies $\P(\Omega') = 1$, so $X(\omega) \in C([0,T]; B(0, R(\omega)))$ and therefore $Q(X(\omega)) \in C([0,T]; H^s)$ for any $s \in [0,3/2)$.) Recall that we assume in particular that $X_0 \in H^{s}$, for any $s \in [0,1/2)$.
 It is then clear that also  $\widetilde{X} \in C([0,T]; H^s)$, $\P$-a.s., for any $s \in [0, 1/2)$. 
 
To get more spatial regularity for $X$ and $\widetilde{X}$ we write 
$$
\widetilde{X}(t) = S_{\theta}(t) X_0 + \widetilde{Y}(t), \quad \text{where $\widetilde{Y}(t) = \frac{1}{2} (R(X^2))(t) + (Q(X))(t)$}.
$$
{\bf Step 3}. We show that $X, \widetilde{X} \in C([0,T]; C(\bar \Lambda))$, $\P$-a.s.

Fix $s = 1/4$. By the continuous inclusion $H^{1/4} \subset L^4(\Lambda)$, we get $X^2 \in C([0,T]; H)$. Using Lemma \ref{lemma:R}(ii) with $u = X^2$, we obtain $R(X^2) \in C([0,T]; H^s)$ for any $s \in [0, 1)$. Using Lemma \ref{lemma:Q} with $u = X$, we obtain $Q(X) \in C([0,T]; H^s)$ for any $s \in [0,3/2)$.

We obtain that $\widetilde{Y} \in C([0,T]; H^s)$, $\P$-a.s., for $s \in [0, 1)$. By a well known Sobolev embedding this implies that $\widetilde{Y} \in C([0,T]; C(\bar \Lambda))$, $\P$-a.s. Using also Step 1 and the fact that $\bar X\in C([0,T]; C(\bar \Lambda))$, $\P$-a.s. by Proposition \ref{prop:regulXbar}, we obtain the assertion.

\vskip 1mm 
\noindent
{\bf Step 4}. We show that $\widetilde{X}$ verifies \eqref{continua}. 
 
We already know that $\bar{X}, \widetilde{X} \in C([0,T]; H^s) \cap C([0,T]; C(\bar \Lambda))$, $\P$-a.s., for any $s \in [0,1/2)$. Since 
$$ 
X^2(t) = (\bar{X}(t) + \widetilde{X}(t))^2, \quad t \in [0,T],
$$ 
we can argue as in the proof of Proposition \ref{prop:algebra2} using Proposition \ref{prop:algebra}. We deduce easily that it is also indeed true for the solution to \eqref{eq:Burgers} that $X^2 \in C([0,T]; H^s) \cap C([0,T]; C(\bar \Lambda))$, $\P$-a.s., for $s \in [0,1/2)$. 

Using Lemma \ref{lemma:R}(iii) with $u = X^2$, we obtain $R(X^2) \in C([0,T]; H^s)$ for any $s \in [0, 3/2)$.
 
Recall that by  Lemma \ref{lemma:Q} we have $Q(X) \in C([0,T]; H^s)$ for any $s \in [0, 3/2)$. We finally obtain that $\widetilde{Y} \in C([0,T]; H^s)$, $\P$-a.s., for $s \in [0, 3/2)$. 
Combining this fact with Step 1 we obtain the assertion. 
\end{proof}

\subsection{Integrability of the nonlinear part} \label{subsection:int}
\begin{proposition} \label{prop:L2 integrability} Let $X_0 \in L^4(\Lambda)$.
The nonlinear part of the solution $\widetilde{X}$ defined in \eqref{mild solution} satisfies $\widetilde{X} \in L^2(\Omega \times [0,T]; C(\bar{\Lambda}))$.
\end{proposition} 
\begin{proof}
We first find an upper bound for $\E \int_0^T \| \widetilde{X}(t) \|_1^2 \, dt$.

We modify the proof of Theorem 14.2.4 in \cite{DPZ96}. In particular we start as in formula \eqref{eq:integrability} in the proof of Proposition \ref{gatarek} with the same notation. We take into account formulae \eqref{eq:integrability} and \eqref{d66}. However we change the estimates in \eqref{s66} as follows:  
\begin{align*}
&\frac{1}{2} \frac{d}{dt} \| \widetilde{X}(t) \|^2 + \theta \| \widetilde{X}(t) \|_1^2 \\
&\leq \| \bar{X}(t) \|_{L^4(\Lambda)} \| \widetilde{X}(t) \|_{L^4(\Lambda)} \| \widetilde{X}(t) \|_1 + \frac{1}{2} \| \bar{X}(t) \|_{L^4(\Lambda)}^2 \| \widetilde{X}(t) \|_1 + C_q \left( 1 + \| \bar{X}(t) \|_{L^q(\Lambda)}^q + \| \widetilde{X}(t) \|^2 \right) \\
&\leq \frac{\theta}{4} \| \widetilde{X}(t) \|_1^2 + \frac{2}{\theta} \| \bar{X}(t) \|_{L^4(\Lambda)}^4 + \frac{2}{\theta} \| \widetilde{X}(t) \|_{L^4(\Lambda)}^4 + \frac{\theta}{4} \| \widetilde{X}(t) \|_1^2 + \frac{1}{\theta} \| \bar{X}(t) \|_{L^4(\Lambda)}^4 \\
&\phantom{\leq}+ C_q + C_q \| \bar{X}(t) \|_{L^q(\Lambda)}^q + C_q \| \widetilde{X}(t) \|^2,
\end{align*}
for a.e. $t \in (0,T].$ Therefore, we find  
$$
\frac{d}{dt} \| \widetilde{X}(t) \|^2 + \theta \| \widetilde{X}(t) \|_1^2 \leq \frac{6}{\theta} \| \bar{X}(t) \|_{L^4(\Lambda)}^4 + \frac{4}{\theta} \| \widetilde{X}(t) \|_{L^4(\Lambda)}^4 + 2 C_q + 2 C_q \| \bar{X}(t) \|_{L^q(\Lambda)}^q + 2 C_q \| \widetilde{X}(t) \|^2
$$
and we integrate it with respect to $t$ in $[0,T]$
\begin{align*}
\| \widetilde{X}(T) \|^2 + \theta \int_0^T \| \widetilde{X}(t) \|_1^2 \, dt &\leq \| X_0 \|^2 + \frac{6}{\theta} \int_0^T \| \bar{X}(t) \|_{L^4(\Lambda)}^4 \, dt + \frac{4}{\theta} \int_0^T \| \widetilde{X}(t) \|_{L^4(\Lambda)}^4 \, dt \\
&\phantom{\leq}+ 2 C_q T + 2C_q \int_0^T \| \bar{X}(t) \|_{L^q(\Lambda)}^q \, dt + 2 C_q \int_0^T \| \widetilde{X}(t) \|^2 \, dt
\end{align*}
(note that the corresponding estimate in the proof of Theorem 14.2.4 in \cite{DPZ96} is different). Now we apply expectation and obtain the following bound
\begin{align}
\E \int_0^T \| \widetilde{X}(t) \|_1^2 \, dt &\leq \frac{\| X_0 \|^2}{\theta} + \frac{6}{\theta^2} \E \int_0^T \| \bar{X}(t) \|_{L^4(\Lambda)}^4 \, dt + \frac{4}{\theta^2} \E \int_0^T \| \widetilde{X}(t) \|_{L^4(\Lambda)}^4 \, dt \notag \\
&\phantom{\leq}+ C + C \E \int_0^T \| \bar{X}(t) \|_{L^q(\Lambda)}^q \, dt + C \E \int_0^T \| \widetilde{X}(t) \|^2 \, dt \label{eq:newbound}
\end{align}
for some constant $C > 0$. We show that the right hand side of \eqref{eq:newbound} is finite. Recall that $\widetilde{X}(t) = {X}(t) - \bar{X}(t) $. We examine separately $X$ and $\bar X$. The fact that  
\begin{equation} \label{eq:L4-bound}
\E \int_0^T \| {X}(t) \|_{L^4(\Lambda)}^4 \, dt < \infty
\end{equation}
\vskip -2mm 
\noindent
follows by Proposition \ref{prop:DPD} with $k = 4$. It is straightforward to check that 
\begin{equation} \label{eq:Lp-bound}
\E \int_0^T \| \bar{X}(t) \|_{L^q(\Lambda)}^q \, dt = \E \int_0^T \int_{\Lambda} | \bar{X}(t,x) |^q \, dx \, dt < \infty,
\end{equation}
for any $q \geq 2$. For instance, one can argue as in Remark 5.2.10 of \cite{DPZ96} which allows to get $\sup_{t \in [0,T]} \E \| \bar{X}(t, \cdot)  \|_{C(\bar \Lambda)}^q < \infty$.

Using \eqref{eq:L4-bound} and \eqref{eq:Lp-bound} with $q = 4$ provides the finiteness of the first two integrals on the right-hand side in \eqref{eq:newbound}.
Analogously, $\E \int_0^T \| {X}(t) \|^2 \, dt < \infty$ follows by Proposition \ref{prop:DPD} (since $\| {X}(t) \|^2 \leq 1 + \| {X}(t) \|_{L^4(\Lambda)}^4$) and
$\E \int_0^T \| \bar{X}(t) \|^2 \, dt $ $< \infty$, by \eqref{eq:Lp-bound} using $q = 2$. The overall result
$$
\E \int_0^T \| \widetilde{X}(t) \|_1^2 \, dt < \infty
$$
\vskip -2mm 
\noindent
means that $\widetilde{X} \in L^2(\Omega \times [0,T]; H^1)$. The Sobolev embedding $H^1 \subset C(\bar{\Lambda})$ finishes the proof.
\end{proof}

\subsection{Proofs of the main theorems} \label{subsection:mainresults}

In this section we prove our main result, that is the central limit theorem (see Theorem \ref{thm:main}).

We consider the error decomposition \eqref{eq:error}. While the asymptotics of the term $\delta^{-1} (\mathcal I_{\delta})^{-1} \mathcal M_{\delta}$ for $\delta \rightarrow 0$ can be analyzed as in \cite{ACP}, we use different techniques to tackle the term $\delta^{-1} (\mathcal I_{\delta})^{-1} \mathcal R_{\delta}$. However, we start similarly. 

From Proposition \ref{prop:asymptoticsbarI}(i)-(iii), it follows that $\mathcal I_{\delta} = O_{\P}(\delta^{-2})$. Therefore, to obtain $\delta^{-1} (\mathcal I_{\delta})^{-1} \mathcal R_{\delta} \rightarrow 0$, it is sufficient to show that
$$
\delta \| K \|_{L^2(\R)}^2 \mathcal R_{\delta} \stackrel{\P}{\rightarrow} 0, \quad \delta \rightarrow 0.
$$
To ease the notation, we omit the $x_0$ in the lower index and, by integration by parts, we write
\begin{align*}
\| K \|_{L^2(\R)}^2 \mathcal R_{\delta} &= \frac{1}{2} \int_0^T X_{\delta}^{\Delta}(t) \left\langle X^2(t), \partial_x K_{\delta} \right\rangle \, dt + \int_0^T X_{\delta}^{\Delta}(t) \left\langle F(X(t)), K_{\delta} \right\rangle \, dt \\
&= \frac{1}{2} (U_{1, \delta} + U_{2, \delta} + U_{3, \delta}) + U_{4, \delta}, \quad \text{where} \\
U_{1, \delta} &:= \int_0^T \bar{X}_{\delta}^{\Delta}(t) \left\langle \bar{X}^2(t), \partial_x K_{\delta} \right\rangle \, dt, \\
U_{2, \delta} &:= \int_0^T \widetilde{X}_{\delta}^{\Delta}(t) \left\langle X^2(t), \partial_x K_{\delta} \right\rangle \, dt + \int_0^T \bar{X}_{\delta}^{\Delta}(t) \left\langle \widetilde{X}^2(t), \partial_x K_{\delta} \right\rangle \, dt \\
&\phantom{:=}+ 2 \int_0^T \bar{X}_{\delta}^{\Delta}(t) \left\langle \bar{X}(t) (\widetilde{X}(t) - \widetilde{X}(t, x_0)), \partial_x K_{\delta} \right\rangle \, dt =: V_{1, \delta} + V_{2, \delta} + V_{3, \delta}, \\
U_{3, \delta} &:= 2 \int_0^T \widetilde{X}(t, x_0) \bar{X}_{\delta}^{\Delta}(t) \left\langle \bar{X}(t), \partial_x K_{\delta} \right\rangle \, dt, \\
U_{4, \delta} &:= \int_0^T X_{\delta}^{\Delta}(t) \left\langle F(X(t)), K_{\delta} \right\rangle \, dt.
\end{align*}
\vskip -2mm 
\noindent
In the next, we will show that $\delta U_{1, \delta} \stackrel{\P}{\rightarrow} 0$ using the Wick's theorem and $\delta V_{j, \delta} \stackrel{\P}{\rightarrow} 0$, for $j = 1, 2, 3$, using the excess spatial regularity of $\widetilde{X}$ over $\bar{X}$. The convergence $\delta U_{3, \delta} \stackrel{\P}{\rightarrow} 0$ will be established by a~specific representation of $\widetilde{X}(t, x_0)$. For the convergence $\delta U_{4, \delta} \stackrel{\P}{\rightarrow} 0$, we will use the assumption (c) in \eqref{bb}. The upper bounds for relevant terms are established in Lemma \ref{lemma:upperbounds} in Appendix. With that we assert the needed convergences.

\begin{lemma} \label{lemma:U1}
As $\delta \rightarrow 0$, we have that $\delta U_{1, \delta} \stackrel{\P}{\rightarrow} 0$.
\end{lemma}
\begin{proof}
The processes $\bar{X}$ and $\bar{X}_{\delta}^{\Delta}$ are centered Gaussian processes, so we may use the Wick's formula (c.f. Theorem 1.28 in \cite{J}) and repeat the proof as for Lemma S.7 in \cite{ACP} with $\gamma = 0$ using the adjusted Lemmas \ref{lemma:S5ACP} and \ref{lemma:S6ACP}.
\end{proof}

\pagebreak
\begin{lemma} \label{lemma:U2}
As $\delta \rightarrow 0$, we have that $\delta U_{2, \delta} \stackrel{\P}{\rightarrow} 0$.
\end{lemma}
\begin{proof} According to Remark \ref{altri} we change the argument of the proof of Lemma S.4 (see the supplement of \cite{ACP}). 

Lemma \ref{lemma:upperbounds}(i)-(ii) yields $V_{1, \delta} = O_{\P}(\delta^{-1/2 - \eps})$, $V_{2, \delta} = O_{\P}(\delta^{-1/2 - \eps})$ for any small $\eps > 0$. As for the term $V_{3, \delta}$, we show the desired convergence in the $\P$-a.s. sense.

Fix $\omega \in \Omega$, $\P$-a.s. By Proposition \ref{prop:regulXtilde}, see in particular \eqref{continua}, we can consider that $\widetilde{X} \in C((0,T]; H^{3/2 - \eps})$, $\P$-a.s., for any \hbox{small} $\eps \in (0,1/2)$. However, the space $H^{3/2 - \eps}$ continuously embedds into $H^{3/2 - \eps} \subset C^{0, 1 - \eps}(\Lambda) = C^{1 - \eps}(\Lambda)$. 
  
By \eqref{continua} there exists a non-negative Borel function $g_{3/2 - \eps}(t) \in L^2([0,T]) $ and $C(\omega)>0$ such that
\begin{equation} \label{eq:holderXtilde} 
|\widetilde{X}(t, x, \omega) - \widetilde{X}(t, y, \omega)| \leq C(\omega) \, g_{3/2 - \eps} (t) \, |x - y|^{1 - \eps},
\end{equation}  
for any $x, y \in \Lambda$, $t \in (0,T]$. Since $\left| \bar{X}_{\delta}^{\Delta}(t) \right| \lesssim \delta^{-1 - \eps}$, $\P$-a.s., uniformly in $t \in [0,T]$, by Lemma \ref{lemma:upperbounds}, we start with the following upper bound:
\begin{align}
\delta |V_{3, \delta}(\omega)| &\leq 2 \delta \int_0^T \left| \bar{X}_{\delta}^{\Delta}(t, \omega) \right| \left| \left\langle \bar{X}(t, \omega) (\widetilde{X}(t, \omega) - \widetilde{X}(t, x_0, \omega)), \partial_x K_{\delta} \right\rangle \right| \, dt \notag \\
&\lesssim \delta^{- \eps} \int_0^T \left| \int_{\Lambda} \bar{X}(t, x, \omega) (\widetilde{X}(t, x, \omega) - \widetilde{X}(t, x_0, \omega)) \partial_x K_{\delta}(x) \, dx \right| \, dt \notag \\
&\lesssim \delta^{- \eps} \int_0^T \left| \int_{\Lambda_{\delta, x_0}} \bar{X}(t, \delta y + x_0, \omega) (\widetilde{X}(t, \delta y + x_0, \omega) - \widetilde{X}(t, x_0, \omega)) \frac{K'(y) \delta}{\delta^{1/2} \delta} \, dy \right| \, dt \notag \\
&\leq \delta^{-1/2 - \eps} \int_0^T \int_{\Lambda_{\delta, x_0}} \left| \bar{X}(t, \delta y + x_0, \omega) \right| \left| \widetilde{X}(t, \delta y + x_0, \omega) - \widetilde{X}(t, x_0, \omega) \right| |K'(y)| \, dy \, dt \notag \\
&\lesssim \delta^{1/2 - \eps - \eps}  \int_0^T  g_{3/2 - \eps} (t) \, \int_{\Lambda_{\delta, x_0}} |y|^{1 - \eps} |K'(y)| \, dy \, dt, \label{eq:V3}
\end{align}
where we used \eqref{eq:holderXtilde} and the fact that $ \left| \bar{X}(t, \delta y + x_0, \omega) \right| \rightarrow \left| \bar{X}(t, x_0, \omega) \right|$ as $\delta \rightarrow 0$ by continuity, u\-ni\-form\-ly in $t \in [0,T]$ (cf. Proposition \ref{prop:regulXbar}); therefore this term is bounded by some positive random constant.

Since the kernel $K$ has compact support, the right-hand side of \eqref{eq:V3} is finite and tends to zero as $\delta \rightarrow 0$ for $\omega \in \Omega$, $\P$-a.s., which finishes the proof.
\end{proof}

Now we consider the most difficult term:
$$
U_{3, \delta} = 2 \int_0^T \widetilde{X}(t, x_0) \bar{X}_{\delta}^{\Delta}(t) \left\langle \bar{X}(t), \partial_x K_{\delta} \right\rangle \, dt.
$$

\begin{lemma} \label{lemma:U3}
As $\delta \rightarrow 0$, we have that $\delta U_{3, \delta} \stackrel{\P}{\rightarrow} 0$.
\end{lemma}
\begin{proof} We cannot adapt the method used in the proof of Lemma S.8 (see supplement of \cite{ACP}) which uses Malliavin calculus. Indeed we do not have a Wiener process $W$ with values in $H = L^2(\Lambda)$.

We argue differently. The main point of the proof will be to represent $\widetilde{X}(t, x_0)$, $x_0 \in \Lambda $, $t\in [0,T]$ a.e., in the form
\begin{equation} \label{eq:repre}
\widetilde{X}(t, x_0) = \sum_{i=1}^{\infty} b_i(t) \nu_i,
\end{equation}
where $b_i \in L^2([0,T])$ are deterministic functions and $(\nu_i)_{i=1}^{\infty}$ form an orthonormal basis in a separable space $L^2(\Omega, \mathcal F', \P)$, with a $\sigma$-field $\mathcal F'\subset \mathcal F$ that will be specified later (the previous series will then converge in $L^2(\Omega, \mathcal F', \P)$).

Then we will show the desired convergence of $\delta U_{3, \delta}$ in $L^1(\Omega)$. To that end, we proceed in several steps. 

\vskip 1mm
{\bf Step 1}. First, recall the formal representation of the cylindrical Wiener process $W(t): W(t) = \sum_{n=1}^{\infty} e_n \beta_n(t)$, where $\beta_n$, $n \geq 1$, are independent \hbox{real} standard Brownian motions and $(e_n)_{n=1}^{\infty}$ is the fixed orthonormal basis of eigenfunctions in $L^2(\Lambda)$ (see Section \ref{subsection:notation}).

We start by considering the regular version of $\bar X$ given by Proposition \ref{prop:regulXbar}. Let us consider 
\begin{equation} \label{frr}
Y_n(t)= \langle \bar X(t), e_n \rangle.
\end{equation}
It is well known that these are one dimensional Ornstein-Uhlenbeck processes:
$$
Y_n(t) = \int_0^t e^{- \vartheta \lambda_n (t-s)} \, d\beta_n(s), \quad t \geq 0, \quad n \geq 1. 
$$
Therefore, we also have, $\P$-a.s, for any $t \geq 0$, $n \geq 1$,  

\begin{equation}\label{stt}
Y_n(t) = -\theta \lambda_n\int_0^t e^{-\theta \lambda_n  (t-s)} \, \beta_n(s) \, ds + \beta_n(t).
\end{equation}

{\bf Step 2}. Let us fix $T > 0$. Example 5.a in \cite{IN} provides the following representation for a real Brownian motion $\beta(t)$:
\begin{equation} \label{eq:itonisio}
\beta(t) = \sum_{k=1}^{\infty} \int_0^t \varphi_k(u) \, du \cdot \int_0^T \varphi_k(u) \, d\beta(u), \quad \P-a.s.,
\end{equation} 
where $(\varphi_k)_{k=1}^{\infty}$ is any orthonormal basis in $L^2([0,T])$. Moreover, Theorem 5.1 in \cite{IN} yields that the infinite series in \eqref{eq:itonisio} converges uniformly in $t \in [0,T]$ to $\beta(t)$, $\P$-a.s. Hence, we have, for any $t \in [0, T]$, $n \geq 1$, $\P$-a.s.,
\begin{align*} 
Y_n(t) &= \sum_{k=1}^{\infty} (-\theta \lambda_n)\int_0^t e^{-\theta \lambda_n  (t-s)}\, \Big(\int_0^s \varphi_k(u) \, du \Big) \, ds \int_0^T \varphi_k(u) \, d\beta_n(u) \\
&\phantom{=}+ \sum_{k=1}^{\infty} \int_0^t \varphi_k(u) \, du \cdot \int_0^T \varphi_k(u) \, d\beta_n(u).
\end{align*}  
\vskip -2mm 
\noindent
Let us define $\xi_{k,n} = \int_0^T \varphi_k(u) \, d\beta_n(u)$, $k,n \geq 1$. Note that $\xi_{k,n}$ are i.i.d. random variables, $\xi_{k,n} \sim N(0,1)$ for every $k,n \geq 1$. Let us define deterministic functions $g_k(t) = \int_0^t \varphi_k(u) \, du$, $t \in [0,T]$. Let $n \geq 1$. We have that
\begin{equation} \label{ma1}
Y_n(t)(\omega) = \sum_{k=1}^{\infty} \Big (-\theta \lambda_n\int_0^t e^{-\theta \lambda_n (t-s)} g_k(s) \, ds + g_k(t) \Big) \xi_{k,n}(\omega).
\end{equation}
It is easy to see that there exists an event $\Omega'$ with $\P(\Omega') = 1$ (independent of $k, n \geq 1$ and $t \in [0,T]$) such that \eqref{ma1} holds for any $\omega \in \Omega'$, $t \in [0,T]$, and the random series converges uniformly on $[0,T]$, $T > 0$. 
 
\vskip 1mm
{\bf Step 3}. We introduce the $\sigma$-field
\begin{equation}\label{separ}
\text{$\mathcal G'$ generated by the random variables $\xi_{k,n}$, $k \geq 1$, $n \geq 1$.}
\end{equation}
Moreover $\mathcal F'$ is obtained completing $\mathcal G'$ by adding all the $\P$-null sets of $(\Omega, \mathcal F, \P)$.

It turns out that by \eqref{ma1}, for any $t \in [0,T]$, $n \geq 1$, 
$$ 
\text{$Y_n(t) = \langle \bar{X}(t), e_n \rangle$ is $\mathcal F'$-measurable.}
$$
Applying Proposition 1.1 in \cite{DPZ} we obtain that, for any $t \in [0,T]$, 
\begin{equation}\label{sdf}
\text{$\bar{X}(t)$ is $\mathcal F'$-measurable with values in $H$.}
\end{equation}

{\bf Step 4}. Using the approach of \cite{gatarekDap} as explained in the proof of Proposition  \ref{gatarek} one can show that, for any $t \in [0,T]$,
\begin{equation}\label{sff}
\text{$\widetilde{X}(t)$ is ${\mathcal F}'$-measurable.}
\end{equation}
To clarify \eqref{sff} consider first the equation \eqref{g12}. Since we can solve it by the contraction principle on a~finite number of ``small'' deterministic time intervals $I_{R,k} \subset [0,T]$, $k = 1, \ldots, N$, we obtain, $\P$-a.s., the solution $X_n$ as limit in $C(I_{R,k}; H)$, $k = 1, \ldots, N$, of approximating processes which at any fixed time are $\mathcal F'$-measurable random variables.

For instance, let us fix $n \geq 1$ and consider the first time interval $[0,R] = I_{R,1}$. The solution $X_n(t)$ can be obtained as limit, $\P$-a.s., in $C(I_{R,1}; H)$ of processes $Z_k(t)$ obtained by the usual iteration scheme (we omit the dependence on $n$ to simplify the notation): 
\begin{align*}
Z_0(t) &= 0, \quad t \in [0,R], \\
Z_{k+1}(t) &= {S_{\theta}(t) X_0 + \frac{1}{2} \int_0^t S_{\theta}(t-s) \partial_x \left( [\pi_n ( Z_{k}(s)) ]^2 \right) \, ds} + \int_0^t S_{\theta}(t-s) F(\pi_n (Z_k(s))) \, ds + \bar{X}(t), \quad t \in [0,R],
\end{align*}
$k \geq 0$. Arguing by induction, it is not difficult to prove that each $Z_k(t)$ is $\mathcal F'$-measurable. Consequently, $X_n(t)$ is $\mathcal F'$-measurable for any $t \in [0,R].$ We can argue similarly on the other intervals $I_{R,k}$.

It follows that, for any $n \geq 1$, $t \in [0,T]$, $X_{n}(t)$ is $\mathcal F'$-measurable. We also note that the solution $X$ to the generalized stochastic Burgers equation verifies, $\P$-a.s., for any $t \in [0,T]$, $n \geq 1$,
$$
X(t \wedge \tau_n) =  X_n(t \wedge \tau_n),
$$
where $\tau_n$ is defined by \eqref{tau}. We deduce that $X(t \wedge \tau_n)$ is $\mathcal F'$-measurable, $t \in [0,T]$, $n \geq 1$.

Passing to the limit, $\P$-a.s. as $n \rightarrow \infty$ (since $\tau_n \rightarrow \infty$) we obtain that for each $t \in [0,T]$ the $H$-valued random variable $X(t)$ is $\mathcal F'$-measurable. Assertion \eqref{sff} follows recalling that $\widetilde{X}(t) = {X}(t) - \bar{X}(t)$.

\vskip 1mm
{\bf Step 5}. Recall that Proposition \ref{prop:L2 integrability} provides $\widetilde{X} \in L^2(\Omega \times [0,T]; C(\bar{\Lambda}))$ and therefore, for any $x_0 \in \Lambda$, $\widetilde{X}(x_0) \in L^2(\Omega \times [0,T], \mathcal F' \otimes \mathcal B ([0,T]))$. As the result, for almost all $t$ (with respect to the Lebesgue measure on $[0,T]$) we have
$$
\widetilde{X}(t,x_0) \in L^2(\Omega, \mathcal F', \P).
$$
Recall the $\sigma$-field $\mathcal G'$ in \eqref{separ}. $\mathcal G'$ is countably generated. This means that there exists a countable set $\mathcal E = (E_n)_{n=1}^{\infty} \subset \mathcal G'$ such that $\sigma(\mathcal E) = \mathcal G'$.

By Proposition 3.4.5 in \cite{C}, $L^2(\Omega, \mathcal G', \P)$ is a separable Hilbert space. By standard arguments of measure theory, we deduce that also
$$   
\text{$L^2(\Omega, \mathcal F', \P)$ is a separable Hilbert space.}
$$ 
Hence there exists an orthonormal basis $(\nu_i)_{i=1}^{\infty}$ (different from $(\xi_{k,n})_{k,n=1}^{\infty}$) such that for any $Y \in L^2(\Omega, \mathcal F', \P)$ we have 
\begin{equation} \label{eq:repreY}
Y = \sum_{i=1}^{\infty} \left\langle Y, \nu_i \right\rangle_{L^2(\Omega)} \nu_i.
\end{equation}

Hence, for almost all $t \in [0,T]$, we have the desired representation
\begin{equation} \label{ci1}
\widetilde{X}(t, x_0) = \sum_{i \ge 1} b_i(t) \nu_i,
\end{equation}
with $b_i \in L^2([0,T])$ being deterministic functions.

\vskip 1mm
{\bf Step 6}. We finish the proof of the convergence in $L^1(\Omega)$ of $\delta U_{3, \delta}$ proceeding analogously as in the second part of the proof of Lemma S.8 in supplement of \cite{ACP}. Consider the representation \eqref{ci1} and, for $N \in \N$, set
$$ 
U_{3, \delta, N} := 2 \sum_{i=1}^N \nu_i \underbrace{\int_0^T b_i(t) \bar{X}_{\delta}^{\Delta}(t) \left\langle \bar{X}(t), \partial_x K_{\delta} \right\rangle \, dt}_{=: s_{i, \delta}} = 2 \sum_{i=1}^N \nu_i s_{i, \delta}.
$$

Fix $\eta > 0$ and choose $N = N(\eta) \in \N$ large enough such that
$$
\int_0^T \E \Big( \widetilde{X}(t, x_0) - \sum_{i=1}^N b_i(t) \nu_i \Big)^2 \, dt = \int_0^T \E \Big( \sum_{i=N+1}^{\infty} b_i(t) \nu_i \Big)^2 \, dt \lesssim \int_0^T \sum_{i=N+1}^{\infty} b_i^2(t) \, dt < \eta.
$$
We decompose $\delta \E |U_{3, \delta}| \leq \delta \E |U_{3, \delta} - U_{3, \delta, N}| + \delta \E |U_{3, \delta, N}|$ and we handle both terms separately. By the Cauchy-Schwarz inequality and Gaussianity (recall that if $Z$ is Gaussian, then $\E |Z|^p \lesssim \left( \E Z^2 \right)^{p/2}$), we obtain
\begin{align*}
\delta^2 \left( \E |U_{3, \delta} - U_{3, \delta, N} | \right)^2 &= \delta^2 \left( \E \left| 2 \int_0^T \sum_{i=N+1}^{\infty} \nu_i b_i(t) \bar{X}_{\delta}^{\Delta}(t) \left\langle \bar{X}(t), \partial_x K_{\delta} \right\rangle \, dt \right| \right)^2 \\
&\leq 4 \delta^2 \int_0^T \E \left( \sum_{i=N+1}^{\infty} \nu_i b_i(t) \right)^2 \, dt \cdot \int_0^T \E \left( \bar{X}_{\delta}^{\Delta}(t) \left\langle \bar{X}(t), \partial_x K_{\delta} \right\rangle \right)^2 \, dt \\
&\lesssim \delta^2 \eta \int_0^T \E \left\langle \bar{X}(t), \Delta K_{\delta} \right\rangle^2 \E \left\langle \bar{X}(t), \partial_x K_{\delta} \right\rangle^2 \, dt \\
&= \delta^{-4} \eta \int_0^T \E \left\langle \bar{X}(t), (\Delta_{\delta}K)_{\delta} \right\rangle^2 \E \left\langle \bar{X}(t), (\partial_x K)_{\delta} \right\rangle^2 \, dt \\   
&\lesssim \eta \| (- \Delta_{\delta})^{1/2} K \|_{L^2(\Lambda_{\delta})}^2 \| (- \Delta_{\delta})^{-1/2} \partial_x K \|_{L^2(\Lambda_{\delta})}^2 \lesssim \eta,
\end{align*}
using the lemmas of Section \ref{subsection:resultsACP} for scaling and a uniform boundedness (in $\delta > 0$) of the two terms on the last line. (Here we also use Assumption \ref{ass:L}.) As the result we obtain
\begin{equation} \label{eq:U3remainder}
\sup_{\delta \in (0,1]} \delta \E |U_{3, \delta} - U_{3, \delta, N}| \lesssim \eta^{1/2}.
\end{equation}

Now we focus on $\delta \E |U_{3, \delta, N}|$. By the Cauchy-Schwarz inequality, we have
$$
\delta \E |U_{3, \delta, N}| = 2 \delta \E \Big| \sum_{i=1}^N \nu_i s_{i, \delta} \Big| \leq 2 \delta \Big( \sum_{i=1}^N \E \nu_i^2 \Big)^{1/2} \cdot \Big( \sum_{i=1}^N \E s_{i, \delta}^2 \Big)^{1/2}.
$$ 
The first sum on the right-hand side in the above expression is of order $N^{1/2}$. Since $N$ is fixed (because small $\eta$ is fixed), this factor only contributes to a~constant. Next, we will prove that
\begin{equation} \label{eq:U3mainpart}
\delta^2 \E s_{i, \delta}^2 \rightarrow 0, \quad i \in \N.
\end{equation}
We take advantage of the symmetry in $t$, $t'$ and we write
$$
s_{i, \delta}^2 = 2 \int_0^T \int_0^t b_i(t) b_i(t') \bar{X}_{\delta}^{\Delta}(t) \bar{X}_{\delta}^{\Delta}(t') \left\langle \bar{X}(t), \partial_x K_{\delta} \right\rangle \left\langle \bar{X}(t'), \partial_x K_{\delta} \right\rangle \, dt' \, dt.
$$
To compute the expectation $\E s_{i, \delta}^2$, we use the Wick's formula, noting that all four random terms $\bar{X}_{\delta}^{\Delta}(t)$, $\bar{X}_{\delta}^{\Delta}(t')$, $\left\langle \bar{X}(t), \partial_x K_{\delta} \right\rangle$ and $\left\langle \bar{X}(t'), \partial_x K_{\delta} \right\rangle$ are Gaussian processes. If we denote the needed products of covariances of pairs in the altering partitions by
\begin{align*}
\rho_{1, \delta}(t, t') &:= \E [\bar{X}_{\delta}^{\Delta}(t) \left\langle \bar{X}(t), \partial_x K_{\delta} \right\rangle] \E [\bar{X}_{\delta}^{\Delta}(t') \left\langle \bar{X}(t'), \partial_x K_{\delta} \right\rangle], \\
\rho_{2, \delta}(t, t') &:= \E [\bar{X}_{\delta}^{\Delta}(t') \left\langle \bar{X}(t), \partial_x K_{\delta} \right\rangle] \E [\bar{X}_{\delta}^{\Delta}(t) \left\langle \bar{X}(t'), \partial_x K_{\delta} \right\rangle], \\
\rho_{3, \delta}(t, t') &:= \E [\left\langle \bar{X}(t), \partial_x K_{\delta} \right\rangle \left\langle \bar{X}(t'), \partial_x K_{\delta} \right\rangle] \E [\bar{X}_{\delta}^{\Delta}(t) \bar{X}_{\delta}^{\Delta}(t')],
\end{align*}
we arrive at
$$
\E s_{i, \delta}^2 = 2 \int_0^T \int_0^t b_i(t) b_i(t') (\rho_{1, \delta}(t, t') + \rho_{2, \delta}(t, t') + \rho_{3, \delta}(t, t')) \, dt' \, dt.
$$
Now \eqref{eq:U3mainpart} follows from this and Lemma \ref{lemma:S6ACP} (i)-(iii). Therefore, we deduce that $\delta \E |U_{3, \delta, N}| < \eta$ for sufficiently small $\delta$ (depending on $N$ and thus on $\eta$). Together with \eqref{eq:U3remainder} and since $\eta$ was arbitrary, we obtain $\delta U_{3, \delta} \stackrel{L^1(\Omega)}{\longrightarrow} 0$.
\end{proof}

\begin{remark} \label{Karhunen}
We can compare our expansion \eqref{ci1} with the classical Karhunen-Lo\`eve expansion (see, for instance, Section 37.5 in \cite{loeve}). Let us fix $x_0 \in \Lambda$. In order to obtain an expansion similar to \eqref{ci1} by the Karhunen-Lo\`eve theorem one needs to check the continuity of the mapping:
$$
(t, t') \mapsto \E [ \widetilde{X}(t, x_0)  \, \widetilde{X}(t', x_0) ]
$$   
on $[0,T] \times [0,T]$. However this property is not clear. We can only prove Proposition \ref{prop:L2 integrability} which gives us an $L^2$-regularity for $\widetilde X(\cdot , x_0)$. We have used such property in order to prove \eqref{ci1}. 
\end{remark}
  
\begin{lemma} \label{lemma:U4}
As $\delta \rightarrow 0$, we have that $\delta U_{4, \delta} \stackrel{\P}{\rightarrow} 0$.
\end{lemma}
\begin{proof}
We show the desired convergence in the $\P$-a.s. sense. Fix $\omega \in \Omega$, $\P$-a.s. By Lemma \ref{lemma:upperbounds} we have that $\left| X_{\delta}^{\Delta}(t) \right| \leq \left| \widetilde{X}_{\delta}^{\Delta}(t) \right| + \left| \bar{X}_{\delta}^{\Delta}(t) \right| \lesssim g_{3/2 - \eps}(t) \,
\delta^{-1 - \eps}$,  for  $t\in (0,T]$,  $\eps \in (0,1/2)$
 with the function  $g_{3/2 - \eps}(t)$ as in  \eqref{continua}. We have
$$
\delta |U_{4, \delta}(\omega)| \leq \delta \int_0^T \left| X_{\delta}^{\Delta}(t, \omega) \right| \left| \left\langle F(X(t, \omega)), K_{\delta} \right\rangle \right| \, dt \lesssim \delta^{- \eps} \int_0^T g_{{3}/{2} - \eps}(t) \int_{\Lambda} |F(X(t, \cdot , \omega))(x)| \, |K_{\delta}(x)| \, dx \, dt.   
$$
Now recall that from the proof of Proposition \ref{prop:regulXtilde} we know that $X \in C([0,T]; C(\bar \Lambda))$, $\P$-a.s. It follows that there exists a bounded set $Q \subset C(\bar \Lambda)$ such that $X(t, \cdot , \omega) \in Q$ for any $t \in [0,T]$.

Using hypothesis (c) in \eqref{bb}, we obtain that $F(Q)$ is a bounded set in $L^{p}(\Lambda)$ for some $p > 2$ (actually, also $Q$ depends on $\omega$). Hence
$$
\sup_{t \in [0,T]} \|F(X(t, \omega)) \|_{L^p(\Lambda)} = M_0 < \infty. 
$$
Recall that $q = \frac{p}{p-1} \in [1,2)$. By the H\"older inequality we infer, for any $t \in [0,T]$,
$$
\int_{\Lambda} |F(X(t, \cdot , \omega))(x)| \,  |K_{\delta}(x)| \, dx \leq M_0 \, \|K_{\delta} \|_{L^q(\Lambda)} \leq C \, \delta^{1/q \, - 1/2},
$$
\vskip -2mm 
\noindent
with $C > 0$ possibly depending on $\omega$. Finally we get $\delta |U_{4, \delta}(\omega)| $ $\lesssim $ $\delta^{- \eps} \delta^{1/q \, - 1/2}$. Since  $1/q \, - 1/2 >0$ we can choose $\eps > 0$ small enough and obtain that $\delta |U_{4, \delta}(\omega)|$ tends to zero as $\delta \rightarrow 0$ for $\omega \in \Omega$, $\P$-a.s. The proof is complete.
\end{proof}

\begin{proof}[Proof of Theorem \ref{thm:main}]
Consider the error decomposition \eqref{eq:error}. The term $\delta^{-1} (\mathcal I_{\delta})^{-1} \mathcal R_{\delta} \stackrel{\P}{\rightarrow} 0$ as $\delta \rightarrow 0$ by Lemmas \ref{lemma:U1}, \ref{lemma:U2}, \ref{lemma:U3} and \ref{lemma:U4} so there is no asymptotic bias from the nonlinearity. For the other term, we write
\begin{equation} \label{eq:error2}
\delta^{-1} (\mathcal I_{\delta})^{-1} \mathcal M_{\delta} = \frac{\mathcal M_{\delta}}{(\mathcal I_{\delta})^{1/2}} \cdot \frac{1}{(\delta^2 \mathcal I_{\delta})^{1/2}}.
\end{equation}
Proposition \ref{prop:asymptoticsbarI}(i)-(iii) provides
$$
\frac{\mathcal I_{\delta}}{\E \bar{\mathcal I}_{\delta}} = \frac{\bar{\mathcal I}_{\delta}}{\E \bar{\mathcal I}_{\delta}} + \frac{\delta^2 o_{\P}(\delta^{-2})}{\delta^2 \E \bar{\mathcal I}_{\delta}} \stackrel{\P}{\rightarrow} 1, \quad \text{as well as} \quad \delta^2 \E \bar{\mathcal I_{\delta}} \rightarrow \frac{T \| K' \|_{L^2(\R)}^2}{2 \vartheta \| K \|_{L^2(\R)}^2}.
$$
Therefore, the first factor in \eqref{eq:error2} converges to $N(0,1)$ in distribution by the standard continuous martingale central limit theorem (e.g., \cite{Ku}, Theorem 1.19, or \cite{LS}, Theorem 5.5.4), while the second factor assembles the stated asymptotic variance. The result follows by applying Slutsky's lemma.
\end{proof}

\begin{appendix}

\section{Appendix}

\subsection{Some results from \cite{ACP}} \label{subsection:resultsACP}

We provide the most pertinent results from \cite{ACP} that are revisited and adjusted to our main case $\gamma = 0$, $d = 1$ and $B = I$. 

We start with the scaling lemmas and we continue with establishing upper bounds and convergences that help us to formulate Proposition \ref{prop:asymptoticsbarI} (that corresponds to Proposition 3 in \cite{ACP}) and the convergences $\delta U_{i, \delta} \stackrel{\P}{\rightarrow} 0$, for $i = 1, 2, 3, 4$, provided by Lemmas \ref{lemma:U1}, \ref{lemma:U2}, \ref{lemma:U3} and \ref{lemma:U4}.

We show the proof only for Lemma \ref{lemma:upperbounds}, which provides the important result that is different from the one that would come from Lemma S.3 in \cite{ACP} just by allowing $\gamma = 0$. Recall that $x_0$ is fixed in $\Lambda = (0,1)$.

As in Section S.1.2 in \cite{ACP}, we denote $(\lambda_k, \Phi_k)_{k=1}^{\infty}$ the eigensystem of the positive, self-adjoint op\-er\-a\-tor $- \Delta$ on the shifted domain $\Lambda - x_0$ with Dirichlet boundary conditions. That is
$$
\lambda_k = \pi^2 k^2, \quad \Phi_k(x) = e_k(x + x_0) = \sqrt{2} \sin(\pi k (x + x_0)), \quad x \in (- x_0, 1 - x_0).
$$
To ease the notation, we omit the $x_0$ in the lower index, so we write $\Lambda_{\delta}$, $z_{\delta}$, $\Delta_{\delta}$ instead of $\Lambda_{\delta, x_0}$, $z_{\delta, x_0}$, $\Delta_{\delta, x_0}$. Without loss of generality, we consider $\theta = 1$ for the semigroup and the rescaled semigroup, so we write $S$ and $S_{\delta}$ instead of $S_{\theta = 1}$ and $S_{\theta = 1, \delta, x_0}$.

\begin{lemma}[Lemma 15 in \cite{ACP}] \label{lemma:15ACP}
For $2 \leq p < \infty$, $\delta > 0$:
\begin{itemize}
\item[(i)] If $z \in W^{2,p}(\Lambda_{\delta})$ then $\Delta z_{\delta} = \delta^{-2} \left( \Delta_{\delta} z \right)_{\delta}$.
\item[(ii)] If $z \in L^p(\Lambda_{\delta})$ then $S(t) z_{\delta} = \left( S_{\delta} (t \delta^{-2}) z \right)_{\delta}$, $t \geq 0$.
\end{itemize}
\end{lemma}

\begin{lemma}[Lemma 16 in \cite{ACP}] \label{lemma:16ACP}
Let $2 \leq p < \infty$, $\delta > 0$. If $h > 0$ and $z \in L^p(\Lambda_{\delta})$ (or $h \leq 0$ and $z \in W^{-2h,p}(\Lambda_{\delta})$) then $(- \Delta)^{-h} z_{\delta} = \delta^{2h} \left( (-\Delta_{\delta})^{-h} z \right)_{\delta}$.
\end{lemma}

\begin{proposition}[Proposition 17 in \cite{ACP}] \label{prop:17ACP}
Let $2 \leq p < \infty$, $t > 0$. Then: 
\begin{itemize}
\item[(i)] For any $h\geq 0$ there exists a universal constant $M_{h}<\infty$ such that \\ $\sup_{t > 0, 0 < \delta \leq 1} \left\| (-t\Delta_{\delta})^{h}S_{\delta}(t) \right\|_{L^{p}(\Lambda_{\delta})} \leq M_{h}$.
\item[(ii)] If $z\in L^p(\R)$ then $S_{\delta}(t)(z|_{\Lambda_{\delta}})\rightarrow e^{t \Delta_{0}} z$ in $L^p(\R)$ as $\delta \rightarrow 0$. 
\end{itemize}
\end{proposition}

\begin{lemma}[Lemma 18 in \cite{ACP}] \label{lemma:18ACP}
Let $2 \leq p < \infty$, $h \geq 0$ and let $z \in L^p(\R)$ have compact support in	$\Lambda_{\delta'}$ for some $\delta' > 0$. Then we have for $\delta \leq \delta'$:
\begin{itemize}
\item[(i)] If $z \in W^{2 \lceil h \rceil, p}(\Lambda_{\delta'})$ then $(-\Delta_{\delta})^{h} z \rightarrow (-\Delta_{0})^{h} z$	in $L^p(\R)$ as $\delta \rightarrow 0$.
\item[(ii)] $\sup_{0 < \delta \leq 1} \left\| (-\Delta_{\delta})^{-h} z \right\|_{L^p(\Lambda_{\delta})} \lesssim \max(\| z \|_{L^1(\R)}, \| z \|_{L^p(\R)})$, $h < \frac{1}{2} \left( 1 - \frac{1}{p} \right)$.
\end{itemize}
\end{lemma}

\begin{lemma}[Lemma 19 in \cite{ACP}] \label{lemma:19ACP}
Let $u \in W^{r,p}(\Lambda)$, $z \in W^{-r,q}(\Lambda_{\delta})$ for some $\delta > 0$, $r \geq 0$ and $\frac{1}{p} + \frac{1}{q} = 1$. Then 
$$
\left| \left\langle u, z_{\delta} \right\rangle \right| \leq \delta^{r + \frac{1}{2} - \frac{1}{p}} \| u \|_{r,p} \left\| (-\Delta_{\delta})^{-r/2} z \right\|_{L^{q}(\Lambda_{\delta})}.
$$
\end{lemma}

\begin{lemma}[Lemma 20 in \cite{ACP}] \label{lemma:20ACP}
Let $1 < q < \infty$. The following hold true:
\begin{itemize}
\item[(i)] If $r < 1 - 1/q$ then $\sup_{0 < \delta \leq 1} \left\| (-\Delta_{\delta})^{-r/2}K \right\|_{L^q(\Lambda_{\delta})} < \infty$.
\item[(ii)] If $r \leq 0$ then, as $\delta \rightarrow 0$, $(-\Delta_{\delta})^{-r/2}K \rightarrow (-\Delta_{0})^{-r/2} K$ and $(-\Delta_{\delta})^{-r/2} \Delta K \rightarrow (-\Delta_{0})^{-r/2} \Delta K$	in $L^q(\R)$.
\end{itemize}
\end{lemma}

\begin{lemma}[Lemma 22 in \cite{ACP}] \label{lemma:22ACP}
Define
$$
c(t, z, t', z') := \Cov \left( \left\langle \bar{X}(t), z \right\rangle, \left\langle \bar{X}(t'), z' \right\rangle \right), \quad \text{for } z, z' \in L^2(\Lambda).
$$
Let $z, z' \in L^2(\Lambda_{\delta})$ for $\delta > 0$. Set $z^{(\delta)} := (-\Delta_{\delta})^{-1/2} z$ and $z'^{(\delta)} := (-\Delta_{\delta})^{-1/2} z'$. Then
\begin{align*}
\left| c(t, z_{\delta}, t', z'_{\delta}) \right| &\lesssim \delta^2 \left\| z^{(\delta)} \right\|_{L^2(\Lambda_{\delta})} \left\| z'^{(\delta)} \right\|_{L^2(\Lambda_{\delta})}, \quad 0 \leq t' \leq t \leq T, \\
\int_0^t c(t, z_{\delta}, t', z'_{\delta})^2 \, dt' &\lesssim \delta^6 \left\| (-\Delta_{\delta})^{-1/2} z^{(\delta)} \right\|_{L^2(\Lambda_{\delta})}^2 \left\| z'^{(\delta)} \right\|_{L^2(\Lambda_{\delta})}^2.
\end{align*}
\end{lemma}

\begin{lemma}[Lemma S.2 in \cite{ACP}] \label{lemma:S2ACP}
Grant Assumption \ref{ass:L}. Let $1 < q < \infty$, $r \leq 2$. Then, as $\delta \rightarrow 0$, 
\begin{itemize}
\item[(i)] $(-\Delta_{\delta})^{-r/2} \partial_x K \rightarrow (-\Delta_{0})^{-r/2 + 1} L$ in $L^p(\R)$. 
\item[(ii)] $(-\Delta_{\delta})^{-r/2} \Phi_k(\delta \cdot) \partial_x K \rightarrow \Phi_k(0)(-\Delta_{0})^{-r/2 + 1} L$ in $L^2(\R)$. Moreover, \\ $\sup_{0 < \delta \leq 1} \left\| (-\Delta_{\delta})^{-r/2} \Phi_k(\delta\cdot) \partial_{x} K \right\|_{L^2(\Lambda_{\delta})} \lesssim \lambda_{k}^{r'/2}$ for $r' > r$.
\end{itemize}
\end{lemma}

\begin{lemma}[Lemma S.3 in \cite{ACP}] \label{lemma:upperbounds} For any small $\eps \in (0, 1/2)$ we introduce the function  $g_{3/2 - \eps}(t)$ as in \eqref{continua}. Then for all $0 < t \leq T$, $k \geq 1$, $\P$-a.s.:
\begin{itemize}
\item[(i)] $\left| \widetilde{X}_{\delta}^{\Delta}(t) \right| \lesssim g_{3/2 - \eps}(t) \delta^{- 1/2 - \eps}, \, \left| \left\langle \widetilde{X}^2(t), \partial_x K_{\delta} \right\rangle \right| \lesssim g_{3/2 - \eps}(t) \, \delta^{1/2 - \eps}, \, \left| \left\langle \widetilde{X}(t), \Phi_k \right\rangle \right| \lesssim g_{3/2 - \eps}(t) \, \lambda_k^{-3/4 + \eps}$.
\end{itemize}
For any small $\eps \in (0, 1/2)$, uniformly in $0 \leq t \leq T$, $k \geq 1$, $r \leq 1$, $\P$-a.s.:
\begin{itemize}
\item[(ii)] $\left| \bar{X}_{\delta}^{\Delta}(t) \right| \lesssim \delta^{-1 - \eps}, \, \left| \left\langle X^2(t), \partial_x K_{\delta} \right\rangle \right| \lesssim \delta^{- \eps}$,
\item[(iii)] $\left| \left\langle \Phi_k^2, \partial_x K_{\delta} \right\rangle \right| \lesssim \lambda_k^r \delta^{r - 1/2 - \eps}$.
\end{itemize}
\end{lemma}
\begin{proof} The proof is based on Lemma \ref{lemma:19ACP} and the regularity results on $\bar{X}$ and $\widetilde{X}$ established in Propositions \ref{prop:regulXbar} and \ref{prop:regulXtilde}.
  
For any small $\eps \in (0, 1/2)$ and any $p \geq 2$, we have, $\P$-a.s.,
\begin{align}
&\bar{X} \in C([0,T]; W^{1/2 - \eps, p}(\Lambda)), \notag \\
&\widetilde{X} \in C((0,T]; H^{3/2 - \eps}), \quad \text{with} \quad \|\widetilde{X}(t) \|_{3/2 - \eps} \leq C \, g_{3/2 - \eps}(t), \label{eq:regularity}
\end{align}
for $0 < t \leq T$, some (random) constant $C > 0$ and the function $g_{3/2 - \eps}(t)$ from Proposition \ref{prop:regulXtilde}.

For the first inequality, we use Lemma \ref{lemma:15ACP}(i) for scaling and Lemma \ref{lemma:19ACP} with $p = q = 2$, $r=3/2 - \eps$, applied to $ z_{\delta} = \delta^{-2} (\Delta_{\delta} K)_{\delta}$. We obtain
\begin{align*}
\left| \widetilde{X}_{\delta}^{\Delta}(t) \right| &= \left| \left\langle \widetilde{X}(t), \Delta K_{\delta} \right\rangle \right| = \left| \left\langle \widetilde{X}(t), \delta^{-2} (\Delta_{\delta} K)_{\delta} \right\rangle \right| \lesssim \delta^{-2} \delta^{3/2 - \eps} \| \widetilde{X}(t) \|_{3/2 - \eps} \| (- \Delta_{\delta})^{1/4 + \eps/2} K \|_{L^2(\Lambda_{\delta})} \\
&\lesssim \delta^{- 1/2 - \eps} g_{3/2 - \eps}(t).
\end{align*}
The finiteness of the factor $\| (- \Delta_{\delta})^{1/4 + \eps/2} K \|_{L^2(\Lambda_{\delta})}$ is provided by Lemma \ref{lemma:20ACP}. This gives the result, but note that it is a different result that would come from Lemma S.3 in \cite{ACP} just by allowing $\gamma = 0$. (Also in the terms of factor $\delta^{- 1/2 - \eps}$.)

The remaining inequalities are shown analogously. We use a classical result: $W^{s, p}(\Lambda)$ is an algebra for $p>1/s$ (see \cite{stri}). It follows that $\widetilde{X}^2$ still belongs to $C((0,T]; H^{3/2 - \eps})$ with $\|\widetilde{X}^2(t) \|_{3/2 - \eps} \leq C \, g_{3/2 - \eps}(t)$ and for $p > 2$ large enough (depending on $\eps$) we also have
$$
\bar{X}^2 \in C([0,T]; W^{1/2 - \eps, p}(\Lambda)).
$$
This allows to perform the estimates as in \cite{ACP} and to get  the required inequalities involving $\bar X$ and $X$.
\end{proof} 

For $x, x' \in \Lambda$ and $0 \leq t' \leq t \leq T$, we set
\begin{align*}
c_{t,t'}^{\Delta}(x) &:=\E [\bar{X}_{\delta}^{\Delta}(t) \bar{X}(t',x)], \quad &c_{t,t'}(x,x') &:= \E[\bar{X}(t,x) \bar{X}(t',x')], \\
c_{t,t'}^{(1)}(x,x') &:= c_{t',t}^{\Delta}(x) c_{t,t'}^{\Delta}(x'), \quad &c_{t,t'}^{(2)}(x,x') &:=c_{t,t}^{\Delta}(x) c_{t',t'}^{\Delta}(x'). 
\end{align*}

\begin{lemma}[Lemma S.5 in \cite{ACP}] \label{lemma:S5ACP}
The following assertions hold true, with $c_{t, t} = c_{t, t}(x,x)$: 
\begin{itemize}
\item[(i)] $| \left\langle c_{t,t}, \partial_x K_{\delta} \right\rangle | \lesssim \delta^{-\eps}$ and $\left| \int_0^t \left\langle c_{t,t'}, \partial_x K_{\delta} \right\rangle \, dt' \right| \lesssim \delta^{1/2 - \eps}$.
\item[(ii)] $\left| \int_{\Lambda^2} c_{t,t'}(x,x')^{2} \partial_x K_{\delta}(x) \partial_x K_{\delta}(x') \, dx dx'\right| \lesssim \delta^{-\eps}$.
\item[(iii)] $\left| \left\langle c_{t,t}^{\Delta} c_{t',t}^{\Delta}, \partial_x K_{\delta} \right\rangle \right| \lesssim \delta^{-1}$ and $\left| \left\langle c_{t',t'}^{\Delta} c_{t,t'}^{\Delta}, \partial_x K_{\delta} \right\rangle \right| \lesssim \delta^{-1}$.	
\item[(iv)] For $i=1, 2$, as $\delta \rightarrow 0$,
$$
\delta^2 \int_{\Lambda^{2}} \int_0^T \int_0^t c_{t,t'}(x,x') c_{t,t'}^{(i)}(x,x') \partial_x K_{\delta}(x) \partial_x K_{\delta}(x')\, dt' dt d(x,x') \rightarrow 0. 
$$	
\end{itemize}
\end{lemma}

\begin{lemma}[Lemma S.6 in \cite{ACP}] \label{lemma:S6ACP}
For any $0 \leq t, t' \leq T$, we have
\begin{itemize}
\item[(i)] $\left| \E [\bar{X}_{\delta}^{\Delta}(t) \bar{X}_{\delta}^{\Delta}(t')] \right| \lesssim \delta^{-1 -\eps} |t-t'|^{-1/2 + \eps}$, for $t \neq t'$, and $\eps > 0$.  
\item[(ii)] $\E \left[ \left\langle \bar{X}(t), \partial_x K_{\delta} \right\rangle \left\langle \bar{X}(t'), \partial_x K_{\delta} \right\rangle \right] \rightarrow 0$, as $\delta \rightarrow 0$.
\item[(iii)] $\delta \E \left[ \bar{X}_{\delta}^{\Delta}(t') \left\langle \bar{X}(t), \partial_x K_{\delta} \right\rangle \right] \rightarrow 0$, as $\delta \rightarrow 0$.
\end{itemize}
\end{lemma}

\begin{proposition}[Proposition 3 in \cite{ACP}] \label{prop:asymptoticsbarI}
As $\delta \rightarrow 0$, the following asymptotics hold true:
\begin{itemize}
\item[(i)] $\delta^2 \E \bar{\mathcal I_{\delta}} \rightarrow \frac{T \| K' \|_{L^2(\R)}^2}{2 \vartheta \| K \|_{L^2(\R)}^2}$.
\item[(ii)] $\frac{\mathcal I_{\delta}}{\E \bar{\mathcal I}_{\delta}} \stackrel{\P}{\rightarrow} 1$.
\end{itemize}
Moreover, if
$$
\int_0^T \left( \widetilde{X}_{\delta, x_0}^{\Delta}(t) \right)^2 \, dt = o_{\P}(\delta^{-2}),
$$
then
\begin{itemize}
\item[(iii)] $\mathcal I_{\delta} = \bar{\mathcal I}_{\delta} + o_{\P}(\delta^{-2})$.
\end{itemize}
\end{proposition}

\subsection{Examples of nonlinearities $F$} \label{subsection:examples}

We provide additional details on our examples of $F: L^2(\Lambda) = H \to L^1(\Lambda)$.

\begin{example}
Let $g: \R \rightarrow \R$ be a real function such that $g$ is bounded and locally Lipschitz. Let us fix a function $h \in L^{\infty}(\Lambda)$. For $u \in H$, we define
$$
(F(u))(x) = g \left( \| u \| \right) \cdot h(x), \quad x \in \Lambda, \, a.e.
$$
It is not difficult to see that it satisfies our above assumptions. Since (a) and (c) are straightforward, we only check (b):
\begin{align*}
F(u + v)(x) \cdot v(x) &= g \left( \| u + v\| \right) \cdot h(x) \cdot v(x) \leq \| g \|_{L^{\infty}(\Lambda)} \| h \|_{L^{\infty}(\Lambda)} |v(x)| \\
&\leq C \left(1 + |u(x)|^2 + |v(x)|^2 \right), \quad u, v \in H, \, x \in \Lambda, \, a.e.
\end{align*}
Similarly, one can consider examples like $(F(u))(x) = h(x) \int_0^1 g(y) f(u(y)) \, dy$, $x \in \Lambda, \, a.e.$, where $g, h \in L^{\infty}(\Lambda)$ and $f: \R \rightarrow \R$ is Lipschitz and bounded.
\end{example}

\begin{example} \label{13}
This example concerns with Nemytskii type nonlinearities. Let $f: \R \rightarrow \R$ be a function such that $f \in C^1(\R)$ and the following conditions holds: There exists $C>0$, $q \ge 2$, $C_q >0$ such that
\begin{itemize}
\item[(i)] $|f(x)| \leq C (1 + |x|^2), \, x \in \R$,
\item[(ii)] $|f'(x)| \leq C (1 + |x|), \, x \in \R$, 
\item[(iii)] $f(x+y) \, y \leq C_q (1 + |y|^2 + |x|^q), \, x, y \in \R$.
\end{itemize}

We define $F: H \to L^1(\Lambda)$ as follows:
$$
(F(u))(x) = f \left( u(x) \right), \quad u \in H, \, x \in \Lambda, \, a.e.
$$

Note that (iii) is stronger than the well-known one-side Lipschitz condition: There exists $C > 0$ such that 
\begin{equation} \label{ds}
f(y) \, y \leq C(1 + |y|^2), \quad y \in \R.
\end{equation}
This condition has been imposed in \cite{gyongypardoux} to cover a large class of semilinear SPDEs. However such SPDEs do not contain the Burgers nonlinear term and do not treat nonlocal nonlinearities. Their approach uses comparison theorems for SPDEs. \\

We check that this nonlinearity $F$ satisfies our initial assumptions. \\

Let us consider (a). By (i) we know that $F: H \rightarrow L^1(\Lambda)$ is well-defined.

To show the local Lipschitz property of $F$, we use condition (ii). We fix $R > 0$ and the ball $B(0,R)$ in $H$. Using the mean value theorem for $u,v \in B(0,R)$, we have
\begin{align*}
\| F(v) - F(u) \|_{L^1(\Lambda)} &= \int_\Lambda \left| f(v(x)) - f(u(x)) \right| \, dx \\
&= \int_\Lambda \left| \int_0^1 f'(u(x) + s (v(x) - u(x)))  (u(x) - v(x)) \, ds \right| \, dx \\
&\leq C \int_\Lambda \int_0^1 (1 + |u(x) + s (v(x) - u(x))|) |u(x) - v(x)| \, ds \, dx \\
&\leq C \int_\Lambda (1 + 2 |u(x)| + |v(x)|) |u(x) - v(x)| \, dx \\
&\leq C_R \left( \int_\Lambda |u(x) - v(x)|^2 \, dx \right)^{1/2},
\end{align*}
where we used that $u,v \in B(0,R)$ on the last line. This gives (a). The point (b) follows easily from (iii).
 
Finally note that it is easy to obtain (c) by using only condition (i). Indeed if $K = D(0, M)$ is a ball in $L^{\infty}(\Lambda)$, then
$$
\| F(u) \|_{L^p (\Lambda)}^p = \int_{\Lambda} |f(u(x))|^p \, dx \leq C' \int_{\Lambda} \left( 1 + |u(x)|^{2p} \right) \, dx \leq C' (1 + M^{2p}),
$$
for some constant $C' > 0$. It follows that $F(K)$ is a bounded set in $L^{p}(\Lambda)$ and this gives the assertion. 

We have checked that (i), (ii) and (iii) imply our initial assumptions (a), (b) and (c). \\ 

It is useful to introduce some simple conditions on $f \in C^{1}(\R)$ which imply (i), (ii) and (iii). In the next classes of examples we always assume (i) and (ii). On the other hand, we will consider different conditions which imply (iii).

\begin{itemize}
\item[(I)] $f$ verifies (ii) and has at most a linear growth, i.e., there exists $C > 0$ such that
$$
|f(x)| \leq C (1 + |x|), \quad x \in \R.
$$
\end{itemize}
Clearly this condition implies (i). It is also not difficult to show that this condition implies (iii).

\begin{itemize}
\item[(II)] $f$ verifies (i) and (ii) and moreover there exists $M \in \R$ such that
\begin{equation} \label{der}
f'(x) \leq M, \quad x \in \R.
\end{equation}
\end{itemize}
In order to check (iii), we write
$$
f(x+y) \, y  = \left( f(x+y) - f(x) \right) y + f(x) \, y.
$$
Now by (i) we have $|f(x) y| \leq C(1 + |x|^2) |y| \leq C' (1 + |y|^2 + |x|^4)$ for some constant $C' > 0$ and all $x, y \in \R$. By the mean value theorem we get
$$
\left( f(x+y) - f(x) \right) y \leq M y^{2},
$$
which shows (iii). An example of $f$ in this class is
$$
f(x) = - C x |x| + Mx, \quad x \in \R, \quad C > 0.
$$
Note that the function $g(x) = - x |x|$ is decreasing and has at most quadratic growth. Moreover its derivative has at most linear growth.

\begin{itemize}
\item[(III)] Consider $f \in C^{1}(\R)$ which satisfies (ii) and such that there exists $\eta \in (0,1)$ and $C_{\eta} > 0$ with
\begin{equation} \label{delt}
|f(x)| \leq C_{\eta} (1 + |x|^{1+\eta}), \quad x \in \R.
\end{equation}
(Clearly this condition is stronger than (i)). In addition we assume that $f$ verifies the one-sided Lipschitz condition \eqref{ds}.
\end{itemize}

We prove under these conditions that hypothesis (iii) holds for $f$. We write
\begin{align*}
f(x+y) \, y &= f(x+y) \, (x+y) - f(x+y) \, x \leq C(1 + (x+y)^2) + C_{\eta} |x| (1 + |x+y|^{1+\eta}) \\
&\leq C' (1 + |y|^2 + |x|^2) + C' |x|^{2+\eta} + C' |x| \, |y|^{1+\eta},
\end{align*}
for some constant $C' > 0$. By the Young inequality: $|x| \, |y|^{1+\eta} \leq \frac{|y|^2}{p} + \frac{|x|^{q}}{q}$ with $p = \frac{2}{1 + \eta}$ and $q = \frac{p}{p-1}$. Hence (iii) holds. 

For instance we can fix $\eta \in (0,1)$ and consider any bounded $C^1$-function $g$ which is Lipschitz and nonnegative. Then 
$$
f(x) = - x |x|^{\eta} \, g(|x|^{1-\eta}), \quad x \in \R,
$$
verifies our conditions.

It is clear that (ii), \eqref{delt} and \eqref{ds} hold. Note that in the previous example we can also consider $\eta = 0$ and $\eta = 1$ (to treat these cases we use (I) and (II)).
\end{example}

\subsection{The notion of weak solution} \label{subsection:weak solution}

\begin{proof}[Proof of Lemma \ref{lemma:weak solution}]
For the nonlinear part $\widetilde{X}$ we have by \eqref{eq:nonlinear}, $\P$-a.s.,
\begin{equation} \label{eq:weakformulation}
\widetilde{X}(t) = S_{\theta}(t) X_0 + \frac{1}{2} \int_0^t S_{\theta}(t-s) \partial_x \left( (\bar{X}(s) + \widetilde{X}(s))^2  \right) \, ds + \int_0^t S_{\vartheta}(t-s) F \left( \bar{X}(s) + \widetilde{X}(s) \right) \, ds.
\end{equation}
Now we work with $\omega$ fixed (i.e., we fix $\omega \in \Omega'$ for some event $\Omega'$ which verifies $\P(\Omega') = 1$). We first prove that
$$
\left\langle z, \widetilde{X}(t) \right\rangle = \left\langle z, X_0 \right\rangle + \theta \int_0^t \left\langle \Delta z, \widetilde{X}(s) \right\rangle \, ds - \frac{1}{2} \int_0^t \left\langle \partial_x z, (\bar{X}(s) + \widetilde{X}(s))^2 \right\rangle \, ds + \int_0^t \left\langle z, F \left( \bar{X}(s) + \widetilde{X}(s) \right) \right\rangle \, ds
$$
for any $t \in [0,T]$, $z \in H^2$ (this is so-called weak formulation of \eqref{eq:weakformulation}). To this purpose we argue as in \cite{B}, see also Proposition A.6 in \cite{DPZ}.

Let us consider $q \geq 2$ from our hypothesis \eqref{bb}. Let us choose $s \in [0, 1/2)$ such that $H^s$ embedds into $L^q(\Lambda)$. We know  that, $\mathbb{P}$-a.s., $\bar{X}(\cdot , \omega) \in C([0,T]; H^{s})$  by Proposition \ref{prop:regulXbar}. Setting
$$
y(t) := \bar{X}(t, \omega), \quad t \in [0,T],
$$ 
we choose $(\varphi_n) \subset C([0,T]; H^1)$ such that $\varphi_n (t) \rightarrow y(t)$ in $H^s$, uniformly in $t \in [0,T]$ (see, for instance, Corollary 0.1.3 in \cite{L}).  

Denote $B_s(0, N)$ the closed ball in the space $H^s$ with radius $N$, that is $B_s(0, N) = \{ u \in H^s; \| u \|_s \leq N \}$, and note that such ball is included in a ball of the space $L^q(\Lambda)$.

Note that there exists $N = N(\omega) > 0$ such that $y(t) \in B_s(0,N)$  for any $t \in [0,T]$. Therefore, there exists $n_0 = n_0(\omega) > 0$ such that for any $n \geq n_0$: $\varphi_n(t) \in B_s(0, N+1)$, for any $t \in [0,T]$.
 
Then define
$$
\widetilde{X}_n (t) = S_{\theta}(t) X_0 + \frac{1}{2} \int_0^t S_{\theta}(t-s) \partial_x \left( (\varphi_n (s) + \widetilde{X}_n (s))^2  \right) \, ds + \int_0^t S_{\theta}(t-s) F \left( \varphi_n (s) + \widetilde{X}_n(s) \right) \, ds
$$
and note that $\partial_x \left( (\varphi_n (\cdot) + \widetilde{X}_n (\cdot))^2 \right) \in C([0,T]; H)$ (see also Proposition \ref{prop:regulXtilde}). Hence we have
\begin{align}
\left\langle z, \widetilde{X}_n (t) \right\rangle &= \left\langle z, X_0 \right\rangle + \theta \int_0^t \left\langle \Delta z, \widetilde{X}_n (s) \right\rangle \, ds - \frac{1}{2} \int_0^t \left\langle \partial_x z, (\varphi_n (s) + \widetilde{X}_n (s))^2 \right\rangle \, ds \notag \\
&\phantom{=}+ \int_0^t \left\langle z, F \left( \varphi_n (s) + \widetilde{X}_n(s) \right) \right\rangle \, ds. \label{eq:weakformulationn}
\end{align}

Using the assumptions on the nonlinearity $F$ and $n \geq n_0$, we obtain
\begin{align}
&\int_0^1 F \left(\varphi_n(s, \cdot) + \widetilde{X}_n(s, \cdot ) \right)(x) \,  \widetilde{X}_n(s, x) \, dx \leq C_q \int_0^1 \left( 1 + | \varphi_n(s,x) |^q + |\widetilde{X}_n(s, x)|^2 \right) \, dx \notag \\
&\leq C_q \left( 1 + \| \varphi_n(s) \|_{L^q(\Lambda)}^q + \| \widetilde{X}_n(s) \|^2 \right) \leq C_q \left( 1 + (N+1)^q + \| \widetilde{X}_n(s) \|^2 \right) \notag \\
&\leq C_{q, N+1} \left( 1 + \| \widetilde{X}_n(s) \|^2 \right) \label{eq:formulationF},
\end{align}
for some constant $C_{q, N+1} > 0$. 

Following the proof of Theorem 14.2.4 in \cite{DPZ96} with the additional term handled as in \eqref{eq:formulationF}, we use the Gronwall lemma to obtain an upper bound for $\| \widetilde{X}_n(s) \|^2$ (analogous to (14.2.12) in \cite{DPZ96}). We get
\begin{equation} \label{eq:boundonXntilde}
\sup_{n \geq 1} \sup_{t \in [0,T]} \| \widetilde{X}_n (t) \| = C_0 < \infty,
\end{equation}
where $C_0$ depends on $\omega, T$, $\| X_0 \|$, $N$ and $\sup_{n \geq 1} \| \varphi_n \|_{C([0,T]; H^s)}$ which is finite. Note that
\begin{align}
\widetilde{X}_n (t) - \widetilde{X} (t) &= \frac{1}{2} \int_0^t S_{\theta}(t-s) \partial_x \left( (\widetilde{X}_n (s) + \varphi_n (s))^2 - (\widetilde{X}(s) + y(s))^2 \right) \, ds \notag \\
&\phantom{=}+ \int_0^t S_{\theta}(t-s) \left( F \left( \widetilde{X}_n(s) + \varphi_n(s) \right) - F \left( \widetilde{X}(s) + y(s) \right) \right) \, ds. \label{eq:difference}
\end{align}
We first handle the $L^2(\Lambda)$-norm of the first term on the right-hand side of \eqref{eq:difference}. Arguing as in the proof of Lemma 14.2.1 in \cite{DPZ96} (see also page 393 in \cite{DPDT}) we obtain, for $0< s \leq t $, $t \in (0,T]$,
$$
\left\| S_{\theta}(t-s) \partial_x \left( (\widetilde{X}_n (s) + \varphi_n (s))^2 - (\widetilde{X}(s) + y(s))^2 \right) \right\| \leq \frac{C_1} {(t-s)^{3/4} } \left\| (\widetilde{X}_n(s) + \varphi_n (s))^2 - (\widetilde{X}(s) + y(s))^2 \right\|_{L^1(\Lambda)} 
$$
for some constant $C_1 > 0$. We are using that
$$
\| S_{\theta}(t) \partial_x  \phi\| \leq \frac{C_1}{t^{3/4}} \| \phi\|_{L^1(\Lambda)}, \quad \phi \in H^1, \quad t\in (0,T]. 
$$
Now it is not difficult to see that, using also the bound \eqref{eq:boundonXntilde},
$$
\left\| (\widetilde{X}_n(s) + \varphi_n (s))^2 - (\widetilde{X}(s) + y(s))^2 \right\|_{L^1(\Lambda)} \leq C_2 \| \widetilde{X}_n (s) - \widetilde{X} (s) \| + C_3 \| \varphi_n - y \|_{C([0,T]; H)},
$$
where constants $C_2, C_3 > 0$ are independent of $n$ (they depend also on $\omega$ and $T$). Recall that $\bar{X}(t, \omega) = y(t)$.

Now we handle the $L^2(\Lambda)$-norm of the second term on the right-hand side of \eqref{eq:difference} using the local Lipschitz property of $F$. In order to do that, consider $n \geq n_0$ and note that $\varphi_n(s) + \widetilde{X}_n(s) \in B(0, C_0 + N + 1)$, $y(s) \in B(0, N)$ (we are using closed balls in $H$) and $\widetilde{X}(s) \in B(0, K)$ for some $K > 0$, since $\omega \in \Omega'$ is fixed and we know that $\widetilde{X} \in C([0,T]; C(\bar \Lambda))$, $\P$-a.s., by Proposition \ref{prop:regulXtilde}.
\begin{align*} 
&\left\| S_{\theta}(t-s) \left( F \left( \widetilde{X}_n(s) + \varphi_n(s) \right) - F \left( \widetilde{X}(s) + y(s) \right) \right) \right\| \\
&\lesssim \frac{1}{(t-s)^{1/4}} \left\| F \left( \widetilde{X}_n(s) + \varphi_n(s) \right) - F \left( \widetilde{X}(s) + y(s) \right) \right\|_{L^1(\Lambda)} \\
&\leq \frac{C_{C_0 + N + 1 + N + K}}{(t-s)^{1/4}} \left\| \left( \widetilde{X}_n(s) + \varphi_n(s) \right) - \left( \widetilde{X}(s) + y(s) \right) \right\| \\
&\leq \frac{C_4}{(t-s)^{1/4}} \| (\widetilde{X}_n(s) - \widetilde{X}(s) \| + \frac{C_5}{(t-s)^{1/4}} \| \varphi_n - y \|_{C([0,T]; H)},
\end{align*}
for some constants $C_4, C_5 > 0$ independent of $n$ (they depend also on $\omega$ and $T$).

Using the above upper bounds, we receive
\begin{align*}
\| \widetilde{X}_n (t) - \widetilde{X} (t) \| &\leq C_6 \int_0^t \left( \frac{1}{(t-s)^{3/4}} + \frac{1}{(t-s)^{1/4}} \right) \| (\widetilde{X}_n(s) - \widetilde{X}(s) \| \, ds + C_7 \| \varphi_n - y \|_{C([0,T]; H)}, \\
&\leq C_8 \int_0^t \frac{1}{(t-s)^{3/4}} \| (\widetilde{X}_n(s) - \widetilde{X}(s) \| \, ds + C_7 \| \varphi_n - y \|_{C([0,T]; H)},
\end{align*}
for some constants $C_6, C_7, C_8 > 0$ independent of $n$ (they depend also on $\omega$ and $T$). Applying the Henry-Gronwall lemma (cf. Section 1.2.1 in \cite{Henry}), we find, for any $t \in [0,T]$,
$$
\| \widetilde{X}_n (t) - \widetilde{X} (t) \| \leq C_7 \, M \| \varphi_n - y \|_{C([0,T]; H)},
$$
where $M = M(C_8, T) > 0$. We obtain that, $\P$-a.s.,
$$
\lim_{n \rightarrow \infty} \| \widetilde{X}_n (t) - \widetilde{X}(t) \|_{C([0,T]; H)} = 0.
$$
Now it is easy to pass the limit in \eqref{eq:weakformulationn} as $n \rightarrow \infty$ and get, $\P$-a.s.,
\begin{align}
\left\langle z, \widetilde{X} (t) \right\rangle &= \left\langle z, X_0 \right\rangle + \theta \int_0^t \left\langle \Delta z, \widetilde{X}(s) \right\rangle \, ds - \frac{1}{2} \int_0^t \left\langle \partial_x z, (\bar{X}(s) + \widetilde{X}(s))^2 \right\rangle \, ds \notag \\
&\phantom{=}+ \int_0^t \left\langle z, F \left( \bar{X}(s) + \widetilde{X}(s) \right) \right\rangle \, ds. \label{eq:z,Xtilde}
\end{align}
On the other hand, we have, $\P$-a.s., for any $z \in H^2$, $t \in [0,T]$
\begin{equation} \label{eq:z,Xbar}
\left\langle z, \bar{X} (t) \right\rangle = \theta \int_0^t \left\langle \Delta z, \bar{X} (s) \right\rangle \, ds + \left\langle z, W(t) \right\rangle.
\end{equation}
Summing up equalities \eqref{eq:z,Xtilde} and \eqref{eq:z,Xbar} gives the result, since $X = \bar{X} + \widetilde{X}$.
\end{proof}

\subsection{Rate optimality} \label{subsection:rate optimality}

\begin{proof}[Proof of Proposition \ref{prop:lowerbound}]
It is enough to find a lower bound for
\begin{equation} \label{eq:lowerbound2}
\inf_{\hat{\theta}} \sup_{\kappa \in \{ \kappa_1, \kappa_2 \}} \left( \E_{\kappa} \left( \hat{\theta} - \theta \right)^2 \right)^{1/2}
\end{equation}
for two suitable admissible alternatives $\kappa_1, \kappa_2 \in \Theta$. We fix $\theta_0 > 0$ and consider for some $c > 0$ the alternatives $\kappa_1 = (\theta_0, 0, 0, X_0^{\theta_0})$ and $\kappa_2 = (\theta_0 + c \delta, 0, 0, X_0^{\theta_0 + c \delta})$, where both nonlinearities are nullified and the initial conditions are taken random as $X_0^{\theta} = \int_{- \infty}^0 S_{\theta}(-s) \, dW(s)$ with $\theta \in \{ \theta_0, \theta_0 + c \delta \}$ and the process $(W(s), s \in \R)$ that is now a two-sided cylindrical Wiener process.

It is not difficult to prove that the random variable $X_0^{\theta}$ verifies the assumptions corresponding to \eqref{eq:ini}. We only note that to prove that, for $\omega \in \Omega$, $\P$-a.s., $X_0^{\theta}(\omega) \in  C(\bar \Lambda)$ one can proceed with the same method of the proof of Lemma 5.1 in \cite{DPZ} which is based on the Kolmogorov-Chentsov test (see also Theorem 5.22 in \cite{DPZ}).
 
This choice leads us to a linear submodel of our model \eqref{eq:Burgers} that we consider from now on. That is
\begin{equation} \label{s55} 
\begin{cases}
dX(t) &= \theta \partial_{xx}^2 X(t) \, dt + \, dW(t), \quad X(0) = X_0^{\theta}, \\
X(t)|_{\partial \Lambda} &= 0, \quad 0 < t \leq T,
\end{cases} 
\end{equation}
with $\theta \in \{ \theta_0, \theta_0 + c \delta \}$. The unique mild solution to equation \eqref{s55} is
\begin{equation} \label{eq:llb}
X(t) = S_{\theta}(t) X_0^{\theta} + \int_0^t S_{\theta} (t-s) \, dW(s) =: \widetilde{X}(t) + \bar{X}(t).
\end{equation}
With this initial condition the process $X$ in \eqref{eq:llb} is stationary. We also observe that the process
$$   
X_{\delta, x_0}(t) = \int_{- \infty}^t \left\langle S_{\theta}(t-s) K_{\delta, x_0}, dW(s) \right\rangle 
$$   
is stationary as well. Therefore, we may argue exactly the same as in the proof of Proposition 5.12 and Lemma A.1 in \cite{AR}, so we obtain for sufficiently small $\delta$ and a suitable constant $C > 0$ the lower bound for \eqref{eq:lowerbound2}
$$
C \delta \, \frac{\left\| \left( I - \Delta_{\delta, x_0} \right)^{-1} K \right\|_{L^2(\Lambda_{\delta, x_0})}^2}{\sqrt{T} \left\| \left( - \Delta_{\delta, x_0} \right)^{1/2} K \right\|_{L^2(\Lambda_{\delta, x_0})}^2}.
$$
The $L^2(\R)$-convergence (and the finiteness) of the terms $\left( I - \Delta_{\delta, x_0} \right)^{-1} K$ and $\left( - \Delta_{\delta, x_0} \right)^{1/2} K$ is established using Lemma \ref{lemma:20ACP} and Lemma A.1 in \cite{AR} (cf. also the proof of Theorem 6 in \cite{ACP}). The proof is complete.
\end{proof}

\subsection{The algebra property of Sobolev spaces} \label{subsection:algebra}
This section is used in the proof of Proposition \ref{prop:regulXtilde}. It is well known that the space $H^s$ is an algebra with respect to pointwise multiplication for $s > 1/2$ (in our case $d = 1$). For bounded functions it holds true for any $s \geq 0$. We believe the next result is known, but we have not found a precise reference. We provide a~proof for the sake of completeness.

\begin{proposition} \label{prop:algebra}
For any $s \in [0, 1]$, we have
\begin{equation} \label{eq:algebra}
\| uv \|_s \leq \| u \|_{\infty} \| v \|_s + \| v \|_{\infty} \| u \|_s,
\end{equation}
for any $u, v \in H^s \cap L^{\infty}(\Lambda)$.
\end{proposition}
\begin{proof}   
We will use the theory of Dirichlet forms, see, e.g., \cite{FOT}. 

Consider the negative definite, self-adjoint linear operator $A = \Delta$ with Dirichlet boundary conditions on $\Lambda$. It is well known that the semigroup $e^{t A} = T_t$ is symmetric and Markovian on $L^2(X, m) = L^2(\Lambda)$ (according to the definition used in Theorem 1.4.1 of \cite{FOT} or the one in \cite{H}). We will consider the Dirichlet form $\mathcal E$ associated to $T_t$. By Theorem 1.3.1 in \cite{FOT}, there is a~biunivocal correspondence: $\mathcal E \leftrightarrow e^{tA} \leftrightarrow A$. Moreover,
$$
\Dom (\mathcal E) = \Dom (\sqrt{-A}), \quad \mathcal E(u,v) = \left\langle \sqrt{-A}u, \sqrt{-A}v \right\rangle.
$$

Recall Theorem 1.4.2 in \cite{FOT}: For any $u, v \in \Dom (\mathcal E) \cap L^{\infty}(X,m)$ (in our case $L^{\infty}(X,m) = L^{\infty}(\Lambda)$), we have
$$
\sqrt{\mathcal E(uv, uv)} \leq \| u \|_{\infty} \sqrt{\mathcal E(v,v)} + \| v \|_{\infty} \sqrt{\mathcal E(u,u)}.
$$
By Remark 1 on page 417 in \cite{H}, we know that for any $0 < s \leq 1$, the operator $-(-A)^{s}$ is also the generator of a symmetric Markovian semigroup $T_t^{s}$ on $L^2(\Lambda)$ obtained by subordination of order $s$ of $T_t$. (For more on subordination like $T_t^f = \int_0^{\infty} T_s \, \mu_t^f (ds)$, see for instance \cite{ARu}.)

The Dirichlet form associated to $T_t^{s}$ is $\mathcal E_{s}$ with $\Dom (\mathcal E_{s}) = \Dom \left( (-A)^{s/2} \right)$. We will apply Theorem 1.4.2 of \cite{FOT} to the Dirichlet form $\mathcal E_{s}$. We find
$$
\sqrt{\left\langle (-A)^{s/2}(uv), (-A)^{s/2}(uv) \right\rangle} \leq \| u \|_{\infty} \| (-A)^{s/2} v \| + \| v \|_{\infty} \| (-A)^{s/2} u \|,
$$
for any $u, v \in \Dom (\mathcal E_{s}) \cap L^{\infty}(X,m) = H^{s} \cap L^{\infty}(\Lambda)$ for $s \in (0,1]$. Since the inequality \eqref{eq:algebra} holds elementarily for $s = 0$, the proof is complete. 
\end{proof}

\end{appendix}

\section*{Acknowledgment} This research activity was carried out as part of the PRIN 2022 project ``Noise in fluid dynamics and related models''. The second author is also member of the ``Gruppo Nazionale per l'Analisi Matematica, la Probabilit\`a e le loro Applicazioni (GNAMPA)'', which is part of the ``Istituto Nazionale di Alta Matematica (INdAM)''.

%

\end{document}